\documentclass[reqno, colorBG]{amsart}
\usepackage{amssymb}
\usepackage{amsmath}
\usepackage[mathscr]{euscript}

\usepackage[small]{caption}
\usepackage{psfrag}
\usepackage{graphicx}
\graphicspath{ {images/} }
\usepackage{color}

\makeatletter
\@addtoreset{equation}{section}
\makeatother

\newtheorem{theorem}{Theorem}[section]
\newtheorem{lem}[theorem]{Lemma}

\newtheorem{cor}[theorem]{Corollary}
\newtheorem{rem}[theorem]{Remark}

\newcounter{as}[section]

\newcommand{\mc}[1]{{\mathcal #1}}
\newcommand{\mf}[1]{{\mathfrak #1}}
\newcommand{\mb}[1]{{\mathbf #1}}
\newcommand{\bb}[1]{{\mathbb #1}}

\begin{document}

\title[Asymptotics for scaled KS equation with General Potential]{Asymptotics for scaled Kramers-Smoluchowski equations in several dimensions with general potentials}

\author{Insuk Seo and Peyam Tabrizian}

\address{\noindent Insuk Seo, Seoul National University, Department of Mathematical Sciences and Research
  Institute of Mathematics, GwanAk Ro 1, Seoul 08826, Republic of Korea.
 \newline
  e-mail: \rm \texttt{insuk.seo@snu.ac.kr} }

\address{\noindent Peyam Tabrizian, University of California, Irvine, Department of Mathematics, 340 Rowland Hall, Irvine,
CA 92697, USA.
  \newline e-mail: \rm \texttt{ptabrizi@uci.edu} }



\begin{abstract}
In this paper, we generalize the results
of Evans and Tabrizian \cite{ET}, by deriving asymptotics for the
time-rescaled Kramers-Smoluchowski equations, in the case of a general non-symmetric potential function with multiple wells. The asymptotic limit is described by a system of reaction-diffusion equations whose coefficients are determined by the Kramers constants at the saddle points of the potential function and the Hessians of the potential function at global minima.
\end{abstract}

\maketitle

\section{Introduction}

In this paper, we consider the following Kramers-Smoluchowski equation
\begin{equation}
\begin{cases}
\tau_{\epsilon}\left(\rho_{t}^{\epsilon}-a\Delta_{x}\rho^{\epsilon}\right)\;=\;\textup{div}\left[D\rho^{\epsilon}+\frac{1}{\epsilon}\rho^{\epsilon}D\Phi\right] & \mbox{in }U\times\mathbb{R}^{d}\times[0,T]\;,\\
\rho^{\epsilon}\;=\;\rho_{0} & \mbox{on }U\times\mathbb{R}^{d}\times\{t=0\}\;,
\end{cases}\label{p01-1}
\end{equation}
where $\epsilon>0$ is a scaling parameter, $\rho^{\epsilon}=\rho^{\epsilon}(x,\xi,t)$
is the chemical density, and $\Phi=\Phi(\xi)$ is a smooth potential
function on $\mathbb{R}^{d}$ with multiple wells. This PDE
models a simple chemical reaction on the atomic level. For more information
on the chemical background, consult \cite{PSV1, T} and the references
therein.

Our primary concern is the limiting behavior of $\rho^{\epsilon}$
when $\epsilon$ tends to $0$. In this paper, we show that
the asymptotic limit of $\rho^{\epsilon}$ satisfies a system of reaction-diffusion equations. See Theorem
\ref{t12} for the rigorous formulation of this result.

The one-dimensional case $d=1$ has already been investigated in \cite{PSV1,PSV2, HN, AMPSV}. In those works, $\Phi$ is assumed to be an even potential function with two wells,
and the limit of $\rho^{\epsilon}$ is derived using tools
such as $\Gamma$-convergence \cite{PSV1,PSV2}, a Raleigh-type dissipation functional \cite{HN},
and a Wasserstein gradient flow \cite{AMPSV}. We refer to \cite{ET, T} for more
detailed survey of the history of the one-dimensional problem.

 In \cite{ET}, Evans and Tabrizian developed
a new and direct approach for this problem, based on
a clever test function that satisfies an elliptic PDE, as well as using capacity estimates from \cite{BEGK}. The techniques in \cite{ET} are robust enough to be generalized in higher dimensions, where
$\Phi$ is a double-well potential on $\mathbb{R}^{d}$. The limitation, however, is that it only works for the case where $\Phi$ is symmetric. In this paper, we remove the symmetry-assumption
 and further allow $\Phi$ to have more than two wells. In that case, our analysis becomes more
delicate, and requires a generalized version of variational
principle in \cite{ET}, which is Theorem \ref{th41} of the current paper. We note that a similar result to the current paper has recently been derived by \cite{MZ}
using tools from semiclassical analysis.

We would like to emphasize that the tools developed in Theorem \ref{t41} are also useful for analyzing metastable random processes, which are processes with multiple stable equilibria. It has been noted in \cite{L1, L2, LS, RS} that, by investigating the inhomogeneous version of our main theorem (Theorem \ref{t41}), one can obtain a complete analysis of the metastability of such processes. In addition, this method turns out to be extremely effective in the investigation of metastable diffusions. Two recent papers  \cite{LS, RS} obtained scaling limits of metastable diffusions known as small random perturbations of dynamical systems. Although such a scaling limit has already been developed for a wide class of metastable Markov chains, it has not been previously known for metastable diffusions.

Our paper is organized as follows: In Section \ref{s2}, we introduce the detailed
model and our assumptions on $\Phi$, as well as the main result of this paper. In Section \ref{s3}, we derive some preliminary estimates, in Section \ref{s4} we state and prove the generalized
variational principle mentioned above, and in Section \ref{s5} we construct
the auxiliary test function. Finally, Section \ref{s6} contains the proof of our main result.

\section{Model and Main Result\label{s2}}

\subsection{Potential $\Phi$}

Let $\Phi:\bb{R}^{d}\rightarrow\bb{R}$ be a smooth potential function with multiple minima. In this section, we state our assumptions on $\Phi$, and introduce some notation about the structure of its valleys.

First, we assume that $\Phi(\xi)$ grows to $+\infty$ as $|\xi|\rightarrow\infty$.
Furthermore, suppose $\Phi$ has exponentially tight level
sets, meaning that for all $a\ge0$ there exists a constant $C(a)>0$
such that
\begin{equation}
\int_{\{\xi:\Phi(\xi)\ge a\}}e^{-\Phi(\xi)/\epsilon}d\xi \;\le\; C(a)e^{-a/\epsilon}\label{tight}
\end{equation}
for all $\epsilon\in(0,1)$. Note that \eqref{tight}
is achieved if $\Phi$ grows at least linearly as $|\xi|\rightarrow\infty$.
Moreover, as observed in \cite[Assumption H.1]{BEGK},
\eqref{tight} is also valid if
\begin{equation*}
\liminf_{\xi\rightarrow\infty}\left|\nabla\Phi(\xi)\right|
\;=\;
\liminf_{\xi\rightarrow\infty}\left[\,|\nabla\Phi(\xi)|-2\Delta\Phi(\xi)\right]
\;=\;\infty\;.
\end{equation*}

\begin{figure}
  \protect
\includegraphics[scale=0.40]{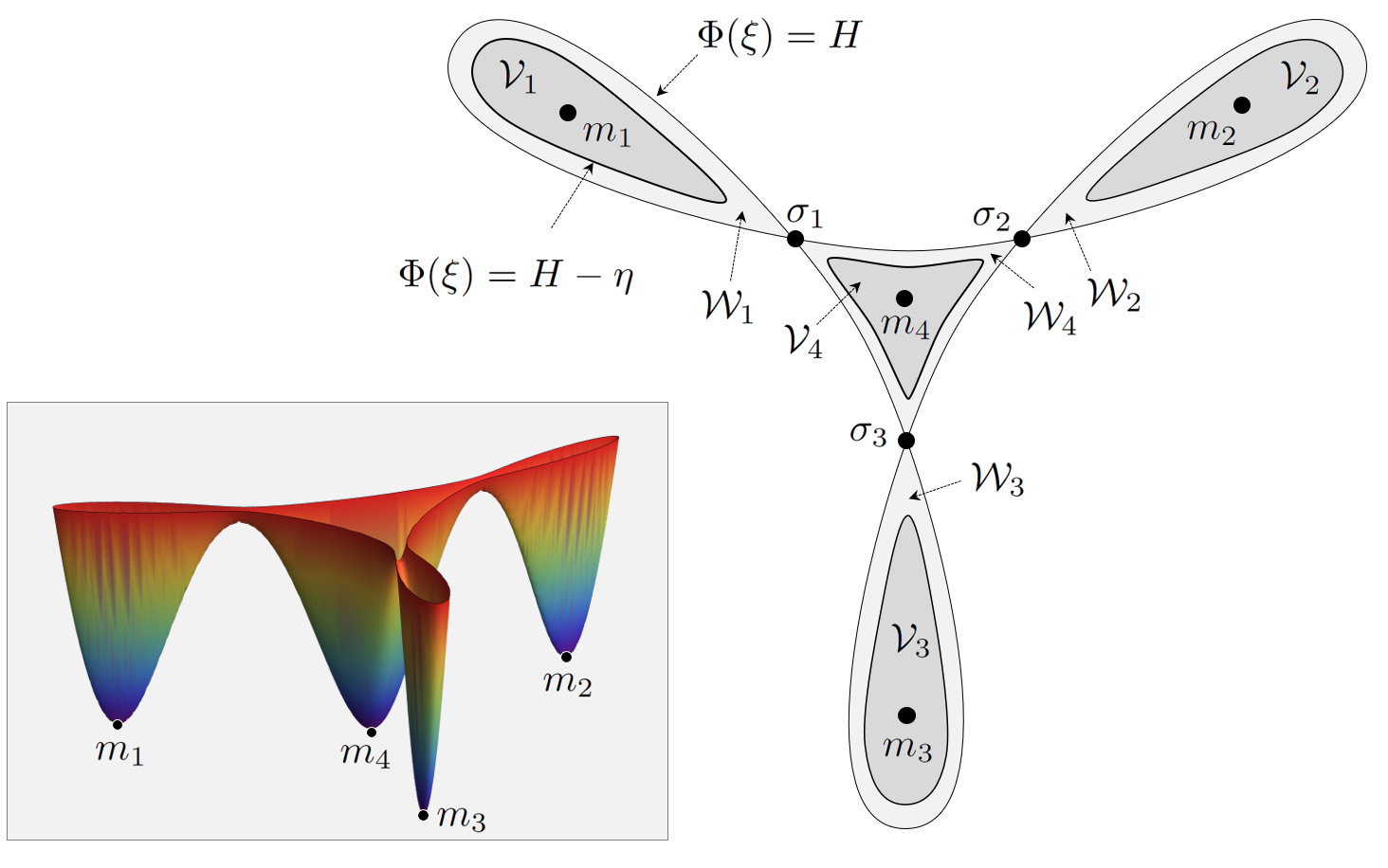}\protect
  \caption{\label{fig1}(Left) An example of the potential function $\Phi$ with four valleys, i.e., $K=4$. (Right) Visualization of the inter-valley structure corresponding to $\Phi$.}
\end{figure}

Now we introduce the inter-valley structure corresponding to the potential function $\Phi$. We refer to Figure \ref{fig1} for the illustration of the definitions below. We will assume that $\Phi$ has finitely many critical points and
achieves minimum at several points. This feature can be characterized more precisely by first defining the valleys of $\Phi$. Fix $H\in\bb{R}$
and let $\mf{\mc{S}}=\{\sigma_{1},\,\sigma_{2},\,\cdots,\,\sigma_{L}\}$
be the set of saddle points of $\Phi$ with height $H$, i.e., $\Phi(\sigma)=H$.

Denote by $\mc{W}_{1},\,\mc{W}_{2},\,\cdots,\,\mc{W}_{K}$
the connected components/valleys of the set $\{\xi:\Phi(\xi)<H\}$. Assume that
$\overline{\mc{W}}_{1}\cup\overline{\mc{W}}_{2}\cup\cdots\cup\overline{\mc{W}}_{K}$
is connected (here $\overline{\mc{A}}$ is the closure of the set $\mc{A}$).

The minimum of $\Phi$ on the valley $\mc{W}_{i}$, $1\le i\le K$,
is achieved at $m_{i}\in\mc{W}_{i}$ and we suppose that
\begin{equation*}
\Phi(m_{1})=\Phi(m_{2})=\cdots=\Phi(m_{K})=h
\end{equation*}
so that valleys $\mc{W}_{1},\,\mc{W}_{2},\,\cdots,\,\mc{W}_{K}$
have the same depth $H-h$. Hence, $m_{1}$, $\cdots$,
$m_{K}$ are minima of $\Phi$.

Let
\begin{equation}\label{sij}
\mc{S}_{i,j}=\overline{\mc{W}}_{i}\cap\overline{\mc{W}}_{j}\subset\mc{S}\;\;;\;1\le i\neq j\le K
\end{equation}
be the set of saddle points between valleys $\mc{W}_{i}$ and
$\mc{W}_{j}$. We select small enough $\eta\in(0,H-h)$ so that
there is no critical point $\xi$ of $\Phi$ such that $\Phi(\xi)\in(H-\eta,\,H)$.
Fix such $\eta$ and define
\begin{equation}
\mc{V}_{i}=\{\xi\in\mc{W}_{i}:\Phi(\xi)<H-\eta\}\;\;;\;1\le i\le K\;.\label{e01}
\end{equation}
Then, the set $\mc{V}_i$, $1\le i \le K$, is connected. Define
\begin{equation}
\Delta=\left(\bigcup_{i=1}^{K}\mc{V}_{i}\right)^{c}\;.\label{del}
\end{equation}
Finally, we assume that, for each saddle point $\sigma\in\mc{S}$,
the Hessian $(D_{\xi}^{2}\Phi)(\sigma)$ has one negative eigenvalue
$-\lambda_{\sigma}$ and $(d-1)$ positive eigenvalues, and for each
minimum $m_{i}$, $1\le i\le K$, the Hessian $(D_{\xi}^{2}\Phi)(m_{i})$
is non-degenerate.

\subsection{Kramers-Smoluchowski equation}

We now describe the scaled Kramers-Smoluchowski equation.
Define
\begin{equation}
\tau_{\epsilon}=\epsilon^{-1}e^{-(H-h)/\epsilon}\;\;\mbox{and}\;\;\sigma^{\epsilon}(\xi)=Z_{\epsilon}^{-1}e^{-\Phi(\xi)/\epsilon}\;,\label{e001}
\end{equation}
where the normalizing factor $Z_{\epsilon}$ is defined by
\begin{equation}
Z_{\epsilon}=\int_{\bb{R}^{d}}e^{-\Phi(\xi)/\epsilon}\,d\xi\label{e0011}
\end{equation}
so that $\int_{\bb{R}^{m}}\sigma^{\epsilon}d\xi=1$. Note that
$Z_{\epsilon}<\infty$ because of \eqref{tight}.

Let $U$ be a bounded, smooth domain in $\bb{R}^{n}$ for some
$n\in\bb{N}$ and let $\frac{\partial\rho^{\epsilon}}{\partial\nu}=D_{x}\rho^{\epsilon}\cdot\nu$
be the outward normal derivative along the boundary $\partial U$.
Let $a:\bb{R}^{n}\rightarrow\bb{R}$ be a smooth and bounded
function such that $a(\cdot)\ge a_{0}>0$ for some constant $a_{0}$.
Fix $T>0$ and consider the equation
\begin{equation}
\begin{cases}
\tau_{\epsilon}\left(\rho_{t}^{\epsilon}-a\,\Delta_{x}\rho^{\epsilon}\right)\;=\;
\textup{div}_{\xi}\left[D_{\xi}\rho^{\epsilon}+\frac{1}{\epsilon}\rho^{\epsilon}\,D_{\xi}\Phi\right] & \mbox{in }U\times\bb{R}^{d}\times[0,T]\;,\\
\frac{\partial\rho^{\epsilon}}{\partial\nu}\;=\;0 & \mbox{on }\partial U\times\bb{R}^{d}\times[0,T]\;,\\
\rho^{\epsilon}\;=\;\rho_{0}^{\epsilon} & \mbox{on }U\times\bb{R}^{d}\times\{t=0\}\;.
\end{cases}\label{p01}
\end{equation}

For $1\le i\le K$, we write
\begin{equation}
\mu_{i}=\frac{1}{\sqrt{\det (D_{\xi}^{2}\Phi)(m_{i})}}\;,\;\;\mu=\sum_{i=1}^{K}\mu_{i}
\;,\;\;\mbox{and}\;\;a_{i}=a(m_{i})\;.\label{e003}
\end{equation}
For $\sigma\in\mc{S}$, denote by $\lambda_{\sigma}$ the unique
negative eigenvalue of the matrix $D_{\xi}^{2}\Phi(\sigma)$, and
define the Kramers constant at $\sigma$ by
\begin{equation*}
\kappa_{\sigma}=\frac{-\lambda_{\sigma}}{2\pi\sqrt{-\det (D_{\xi}^{2}\Phi)(\sigma)}}\;.
\end{equation*}
Recall $\mc{S}_{i,j}$ from \eqref{sij} and define
\begin{equation}
\kappa_{i,j}=\sum_{\sigma\in\mc{S}_{i,j}}\kappa_{\sigma}\;\;;\;1\le i\neq j\le K\;.\label{kappa}
\end{equation}
For convenience we set $\kappa_{i,i}=0$ for all $1\le i\le K.$ Define
the rate constants by
\begin{equation}
r_{i,j}=\frac{\kappa_{i,j}}{\mu_{i}}\;\;;\;1\le i\neq j\le K\;.\label{rij}
\end{equation}

Now we explain our assumptions on the initial data. Consider the normalized initial data
\begin{equation*} u_{0}^{\epsilon}(x,\xi)=\frac{\rho_{0}^{\epsilon}(x,\xi)}{\sigma^{\epsilon}(\xi)}\;.
\end{equation*}
Then, we assume that $u_{0}$ is bounded on $\mathbb{R}$, is differentiable with respect to $x$ and $\xi$, and satisfies
\begin{equation}
\label{init}
\int_{\bb{R}^{d}}\int_{U}\left(|u_{0}^{\epsilon}|^{2}+\left|D_{x}u_{0}^{\epsilon}
\right|^{2}\,+\,
\frac{1}{\tau_{\epsilon}}\left|D_{\xi}u_{0}^{\epsilon}\right|^{2}\right)
\sigma^{\epsilon}\,dxd\xi\;<\;\infty\;.
\end{equation}
Finally, assume that, for smooth functions $\alpha_{1}^{0},\,\cdots,\,\alpha_{K}^{0}:U\rightarrow\bb{R}$, we have the following convergence as $\epsilon$ tends to $0$:
\begin{equation*}
u_{0}^{\epsilon}(x,\xi)\rightarrow\frac{\mu}{\mu_{i}}\alpha_{i}^{0}\;\;\mbox{locally uniformly in }\overline{U}\times\mc{W}_{i}\;\;;\;1\le i\le K\;.
\end{equation*}
Under this set of assumptions, we are now ready to state the main result of our paper:
\begin{theorem}
\label{t12}For all $t\in[0,T]$, we have, in the sense of Remark \ref{rem22},
\begin{equation}
\rho^{\epsilon}(x,\xi,t)\,d\xi \rightharpoonup\sum_{i=1}^{K}\alpha_{i}(x,t)\,\delta_{m_{i}}\;\;\text{as\;\; $\epsilon\rightarrow 0$}\;,\label{wc}
\end{equation}
where the smooth functions $\alpha_{1},\,\cdots,\,\alpha_{K}$
on $U\times[0,T]$ solve the system of linear reaction-diffusion
equations given by
\begin{equation}
\begin{cases}
\partial_{t}\alpha_{i}-a_{i}\Delta\alpha_{i}=\sum_{j=1}^{K}(r_{j,i}\alpha_{j}-r_{i,j}\alpha_{i}) & \mbox{in }U\times[0,T]\\
\frac{\partial\alpha_{i}}{\partial\nu}=0 & \mbox{on }\partial U\times[0,T]\\
\alpha_{i}=\alpha_{i}^{0} & \mbox{on }t=0\;
\end{cases}\label{e1}
\end{equation}
for all $1\le i\le K$.
\end{theorem}
\begin{rem}\label{rem22}
The weak convergence \eqref{wc} means that for all $f = f(x,\xi,t) \in C(U \times A \times [0,T])$,
\begin{equation*}
\lim_{\epsilon\rightarrow 0}\int_{[0,T]} \int_U \int_A  \rho^\epsilon(x,\xi,t) f(x,\xi,t) d\xi dx dt
=\int_{[0,T]}\int_A \sum_{i=1}^{K} \alpha_i (x,t) f(x,m_i,t) dx dt\;.
\end{equation*}
\end{rem}
\subsection{Graph structure of valleys and an associated Markov chain\label{s23}}

The main result described above is closely related
to a Markov chain on a graph whose vertices are the valleys of potential
$\Phi$.
More precisely, denote by $V=\{1,2,\cdots,K\}$ the set of vertices, in such a way that $i\in V$ corresponds to the valley $\mc{V}_{i}$. Moreover, two vertices
$i,\,j\in V$ are connected by an edge if and only if $\overline{\mc{W}}_{i}\cap\overline{\mc{W}}_{j}\neq\phi$,
or equivalently $\kappa_{i,j}\neq0$. Denote by the $G$ the resulting graph. Since we have assumed that the set $\overline{\mc{W}}_{1}\cup\overline{\mc{W}}_{2}\cup\cdots\cup\overline{\mc{W}}_{K}$
is connected, the graph $G$ is a connected graph.

Let $\{X_{t}:t\ge0\}$ be a Markov chain on $V$ where the jump rate
from $i\in V$ to $j\in V$ is $r_{i,j}$ (cf. \eqref{rij}). Since
$r_{i,j}=0$ if $\kappa_{i,j}=0$, $X_{t}$ becomes a
Markov chain on $G$. Define
\begin{equation*}
\widehat{\mu}_{i}=: \frac{\mu_{i}}{\mu}\; \mbox{for }1\le i\le K\;\;\mbox{ and}\;\;\boldsymbol{\mu}:= (\widehat{\mu}_{1},\,\cdots,\,\widehat{\mu}_{K})\;.
\end{equation*}
Then, observe that the probability measure $\boldsymbol{\mu}$ on $V$
is the invariant measure for the Markov chain $X_{t}$, and furthermore,
the Markov chain is reversible with respect to $\boldsymbol{\mu}$
in the sense that $\widehat{\mu}_{i}r_{i,j}=\widehat{\mu}_{j}r_{j,i}$
for all $i \neq j$. The generator $\mc{L}$ of this Markov
chain can be regarded as a linear operator on $\bb{R}^{K}.$ More precisely, for
$\mb{b}=(b_{1},\cdots,b_{K})\in\bb{R}^{K}$, the $i$th component
of $\mc{L}\mb{b}\in\bb{R}^{K}$ is given by
\begin{equation*}
(\mc{L}\mb{b})_{i}\;=\;\sum_{j=1}^{K}r_{i,j}(b_{j}-b_{i}) \;.
\end{equation*}

\begin{rem}
Assume that $a\equiv0$ so that $\alpha_{i}$, $1\le i\le K$,
is a function of time only. Then, define $\widehat{\alpha}_{i}(t)=\alpha_{i}(t)/\widehat{\mu}_{i}$,
and let $\widehat{\boldsymbol{\alpha}}(t)=(\widehat{\alpha}_{1}(t),\cdots,\widehat{\alpha}_{K}(t))\in\bb{R}^{K}$.
Then, we can deduce from \eqref{e1} that
\begin{equation*}
\frac{d\widehat{\alpha}_{i}}{dt}(t)\;=\;\sum_{j=1}^{K}r_{i,j}(\widehat{\alpha}_{j}(t)-\widehat{\alpha}_{i}(t))=\left(\mc{L}\widehat{\boldsymbol{\alpha}}(t)\right)_{i}\;.
\end{equation*}
Therefore, $\boldsymbol{\alpha}(t)=(\alpha_{1}(t),\cdots,\alpha_{K}(t))$
is the marginal density of the Markov chain $X_{t}$ with respect
to the invariant measure $\boldsymbol{\mu}$, whose starting (possibly deterministic) measure
is $(\alpha_{1}^{0},\cdots,\alpha_{K}^{0})$.
\end{rem}

\section{Preliminary Estimates\label{s3}}

In this section, we state and prove estimates. Denote by
$o_{\epsilon}(1)$ the term vanishing as $\epsilon\rightarrow0$.
\begin{lem}
\label{lem31}We have that
\begin{align}
 &\int_{\mc{V}_{i}}e^{-\Phi(\xi)/\epsilon}d\xi\;=\;
 \left[1+o_{\epsilon}(1)\right]e^{-h/\epsilon}(2\pi\epsilon)^{d/2}\mu_{i}\;;\;\;1\le i\le K\;,\label{e11}\\
 &\int_{\Delta}e^{-\Phi(\xi)/\epsilon}d\xi\;=\;o_{\epsilon}(1)e^{-h/\epsilon}
 \epsilon^{d/2}\;,\;\text{and}\label{e12}\\
& Z_{\epsilon}\;=\;\left[1+o_{\epsilon}(1)\right]e^{-h/\epsilon}(2\pi\epsilon)^{d/2}
\mu\;.\label{e13}
\end{align}
\end{lem}
\begin{proof}
The proof of \eqref{e11} is an easy consequence of Laplace's method.
The estimate \eqref{e12} is a direct consequence of \eqref{tight}.
Finally, \eqref{e13} follows immediately from \eqref{e11} and \eqref{e12}
because of the definition of $Z_{\epsilon}$ (cf. \eqref{e0011}). \end{proof}
\begin{lem}
\label{lem32}For $\mc{A}\subset\bb{R}^{d}$, suppose that
there exists $c>0$ such that $\Phi(\xi)\ge H+c$ for all $\xi\in\mc{A}$.
Then,
\begin{equation*}
\int_{\mc{A}}\frac{\sigma^{\epsilon}}{\tau_{\epsilon}}d\xi\;=\;o_{\epsilon}(1)\;.
\end{equation*}
\end{lem}
\begin{proof}
By \eqref{e13},
\begin{equation}
\frac{\sigma^{\epsilon}}{\tau_{\epsilon}}\;=\;\left[1+o_{\epsilon}(1)\right]\,
\frac{\epsilon}{(2\pi\epsilon)^{d/2}\mu}\,e^{(H-\Phi)/\epsilon}\;.\label{e14}
\end{equation}
Hence, the lemma immediately follows from \eqref{tight}.
\end{proof}
Now we establish several compactness estimates similar to \cite[Section 3]{ET}.
Let
\begin{equation*}
u^{\epsilon}(x,\xi,t):=\frac{\rho^{\epsilon}(x,\xi,t)}{\sigma^{\epsilon}(\xi)}\;.
\end{equation*}
Then, by \eqref{p01}, the $u^{\epsilon}$ satisfies
\begin{equation}
u_{t}^{\epsilon}-a\,\Delta_{x}u^{\epsilon}\;=\;\frac{1}{\sigma^{\epsilon}}
\,\textup{div}_{\xi}\left[\frac{\sigma^{\epsilon}}{\tau_{\epsilon}}\,
D_{\xi}u^{\epsilon}\right]\;.\label{p02}
\end{equation}
The next lemma is an energy estimate that is similar to that of \cite[Lemma 3.1]{ET}. However, instead of skipping the proof, we refer the readers to the Appendix, since the notation here is more involved than \cite{ET}.
\begin{lem}
\label{lem21}For some constant $C>0$, we have the bound
\begin{equation}
0\le u^{\epsilon}\le C\;\;\mbox{on }U\times\bb{R}^{d}\times[0,T]\label{ub}
\end{equation}
and the energy estimate
\begin{multline}
\label{uc}
\sup_{0\le t\le T}\int_{\bb{R}^{d}}\int_{U}\left(|u^{\epsilon}|^{2}+|D_{x}u^{\epsilon}|^{2}
+\tau_{\epsilon}^{-1}\left|D_{\xi}u^{\epsilon}\right|^{2}\right)\sigma^{\epsilon}
\,dxd\xi \\
+\int_{0}^{T}\int_{\bb{R}^{d}}\int_{U}|u_{t}^{\epsilon}|^{2}\,
\sigma^{\epsilon}\,dxd\xi dt  \;\le\;  C\;.
\end{multline}
\end{lem}

Define $U_{T}=U\times(0,T)$. We next develop some pre-compactness results similar to \cite[Lemmas 3.2 and 3.3]{ET}. Again, proofs can be found in the Appendix since they are more involved.

\begin{lem}
\label{lem22}There exist a sequence $\{\epsilon_{n}\}_{n=1}^{\infty}$
of positive real numbers converging to $0$ and  functions
$\alpha_{1},\,\alpha_{2},\,\cdots,\,\alpha_{K}\in H^{1}(U_{T})$ that satisfy the following:
\begin{enumerate}
\item For all $1\le i\le K$, we have that, as $n\rightarrow\infty$,
\begin{align}
&\int_{\mc{V}_{i}} \int_{U} \rho^{\epsilon_{n}}(x,\xi,t)\,dxd\xi\,\rightharpoonup\,\alpha_{i}(x,t)\mbox{ \;weakly in }L^{2}(U_{T})\;\mbox{and}\label{ee1}\\
& \sup_{0\le t\le T}\int_{\Delta}|\rho^{\epsilon_{n}}(x,\xi,t)|\,d\xi\,\rightarrow\,0\;.\label{ee2}
\end{align}
\item For all $1\le i\le K$, we have that, as $n\rightarrow\infty$,
\begin{align}
&\int_{\mc{V}_{i}}\partial_{t}\rho^{\epsilon_{n}}(x,\xi,t)\,d\xi  \,\rightharpoonup\,  \partial_{t}\alpha_{i}(x,t)\mbox{ \; weakly in }L^{2}(U_{T})\;,\label{ee3}\\
&\int_{\mc{V}_{i}}D_{x}\rho^{\epsilon_{n}}(x,\xi,t)\,d\xi  \,\rightharpoonup  \, D_{x}\alpha_{i}(x,t)\mbox{ \; weakly in }L^{2}(U_{T})\;,\label{ee31}\\
&\int_{\Delta}\int_{U}|\partial_{t}\rho^{\epsilon_{n}}(x,\xi,t)|\,dxd\xi  \,\rightarrow\, 0\;\;\; \mbox{strongly in \ensuremath{L^{2}(0,T)}}\;,\;\mbox{and}\label{ee32}\\
& \sup_{0\le t\le T}\int_{\Delta}\int_{U}\left|D_{x}\rho^{\epsilon_{n}}(x,\xi,t)\right|\,
dxd\xi\,\rightarrow\,0\;.\label{ee4}
\end{align}
\item For all $t\in[0,T]$, for all $1\le i\le K$, and almost every $x\in U$,
we have that, as $\epsilon\rightarrow0$,
\begin{equation}
u^{\epsilon}(x,\xi,t)\,\rightarrow\,\frac{\alpha_{i}(x,t)}
{\widehat{\mu}_{i}}\;\;\mbox{for almost every }\xi\in\mc{V}_{i}\;.\label{ee5}
\end{equation}
\end{enumerate}
\end{lem}

\section{A variational Problem\label{s4}}

Throughout the rest of the paper, elements of $\bb{R}^{K}$ are denoted by bold lower-case letters such as $\mb{a=}(a_{1},\cdots,a_{K})$, and subsets of $\bb{R}^{K}$
are denoted by bold capital letters like $\mb{A}$ and $\mb{B}$.

Define $\mc{D}:\bb{R}^{K}\rightarrow\bb{R}$
by
\begin{equation}
\mc{D}(\mb{b})\;=\;\frac{1}{2\mu}\sum_{i,j=1}^{K}\kappa_{i,j}(b_{j}-b_{i})^{2}
\;\;;\;\mb{b}\in\bb{R}^{K}\;.\label{diri}
\end{equation}
Note that $\mc{D}(\mb{b})=0$ implies $b_{1}=b_{2}=\cdots=b_{K}$
since the graph $G$ is connected.
\begin{rem}
The function $\mc{D}$ is the so-called Dirichlet form associated
with the generator $\mc{L}$ defined in Section \ref{s23}. More
precisely, we can write
\begin{equation*}
\mc{D}(\mb{b})\;=\;\sum_{i=1}^{K}\widehat{\mu}_{i}\mb{b}_{i}(-\mc{L}\mb{b})_{i}\;.
\end{equation*}
\end{rem}
For $\mb{b}=(b_1,\,\cdots,\,b_K)\in\bb{R}^{K}$, define
\begin{equation}
\mathscr{F}_{\mb{b}}=\left\{ \psi\in H_{1}(\bb{R}^{d}):\psi\big|_{\mc{V}_{i}}\equiv b_{i}\;\mbox{for all}\;1\le i\le K\right\} \;.\label{fb}
\end{equation}
In the current and the next section, we only consider functions on
$\bb{R}^{d}$, that is only depending on $\xi$ and independent of the variable $x$. Hence, for a function $\phi:\mathbb{R}^d\rightarrow \mathbb{R}$, the notations
  $D\phi$ and $\Delta\phi$ are used to
represent $D_{\xi}\phi$ and $\Delta_{\xi}\phi$, respectively. Then the following result is a generalization of \cite[Theorem 3.1]{BEGK}.
\begin{theorem}
\label{th41}For any $\mb{b}\in\bb{R}^{K}$, we have that
\begin{equation}
\inf_{\varphi\in\mathscr{F}_{\mb{b}}}\int_{\bb{R}^{d}}
\frac{\sigma^{\epsilon}}{\tau_{\epsilon}}|D\psi|^{2}d\xi\;=\;
[1+o_{\epsilon}(1)]\,\mc{D}(\mb{b})\;.\label{var2}
\end{equation}
\end{theorem}
\begin{proof}
By \eqref{e14} and definition of $\mc{D}(\cdot)$ we can rewrite the identity
\eqref{var2} as
\begin{equation}
\inf_{\varphi\in\mathscr{F}_{\mb{b}}}
\epsilon\int_{\bb{R}^{d}}e^{-\Phi/\epsilon}|D\psi|^{2}d\xi\;=\;
[1+o_{\epsilon}(1)]\,e^{-H/\epsilon}\,
\frac{(2\pi\epsilon)^{d/2}}{2}\sum_{i,j=1}^{K}\kappa_{i,j}(b_{j}-b_{i})^{2}\;.
\label{vm1}
\end{equation}
Denote by $\varphi_{\mb{b}}^{\epsilon}$ the minimizer of the
left-hand-side. Then, $\varphi_{\mb{b}}^{\epsilon}$ solves the
following Euler-Lagrange equation:
\begin{equation*}
     \mbox{div}\left[e^{-\Phi/\epsilon}\,D\varphi_{\mb{b}}^{\epsilon}\right]\,=\,0
     \;\mbox{on}\;\Delta\;\; \mbox{ and } \; \;  \varphi_{\mb{b}}^{\epsilon}\,=\,b_{i}\;\mbox{on\;}\mc{V}_{i}\;\mbox{for all}\;1\le i\le K\;.
\end{equation*}
For $1\le i\le d$, write $\mb{e}_{i}=(0,\cdots,0,1,0,\cdots,0)$
the $i$th standard basis vector of $\bb{R}^{d}$. Then, by linearity and uniqueness of the Euler-Lagrange equation, it follows that
\begin{equation}
\varphi_{\mb{b}}^{\epsilon}\;=\;\sum_{i=1}^{d}b_{i}\varphi_{\mb{e}_{i}}^{\epsilon}\;.\label{cm1}
\end{equation}
Therefore, we can write
\begin{multline}\label{cham}
\epsilon\int_{\bb{R}^{d}}e^{-\Phi/\epsilon}\,|D\varphi_{\mb{b}}^{\epsilon}|^{2}
\,d\xi\;=\;
 \sum_{i=1}^K\, b_i^2 \,\epsilon\int_{\bb{R}^{d}}e^{-\Phi/\epsilon}|
D\varphi_{\mb{e}_{i}}^{\epsilon}|^{2}\,d\xi
\\
 +\,
\sum_{1\le i\neq j \le K} b_i \,b_j \,\epsilon\int_{\bb{R}^{d}}e^{-\Phi/\epsilon}|\,
D(\varphi_{\mb{e}_{i}}^{\epsilon}+\varphi_{\mb{e}_{j}}^{\epsilon})|^{2}\,d\xi
\end{multline}
In \cite[Theorem 3.1]{BEGK}, it is shown that
\begin{equation}
\epsilon\int_{\bb{R}^{d}}e^{-\Phi/\epsilon}\,|
D\varphi_{\mb{e}_{i}}^{\epsilon}|^{2}d\xi\;=\;[1+o_{\epsilon}(1)]\,
e^{-H/\epsilon}\,(2\pi\epsilon)^{d/2}\,\sum_{l=1}^{K}\kappa_{i,l}\;.\label{vm2}
\end{equation}
and that, for $i\neq j$,
\begin{multline}
\epsilon\int_{\bb{R}^{d}}e^{-\Phi/\epsilon}|\,
D(\varphi_{\mb{e}_{i}}^{\epsilon}+\varphi_{\mb{e}_{j}}^{\epsilon})|^{2}\,d\xi
\\
=\;[1+o_{\epsilon}(1)]\,e^{-H/\epsilon}\,(2\pi\epsilon)^{d/2}\,\sum_{1\le l\le K \,:\,l\neq i,j}(\kappa_{i,l}+\kappa_{j,l})\;.\label{vm3}
\end{multline}
By \eqref{vm2} and \eqref{vm3}, we have that
\begin{equation}
\epsilon\int_{\bb{R}^{d}}e^{-\Phi/\epsilon} \,D\varphi_{\mb{e}_{i}}^{\epsilon}\cdot D\varphi_{\mb{e}_{j}}^{\epsilon}\,d\xi
\;=\;-[1+o_{\epsilon}(1)]\,e^{-H/\epsilon}\,
(2\pi\epsilon)^{d/2}\,\kappa_{i,j}\;.\label{vm4}
\end{equation}
We can complete the proof by combining \eqref{cham}, \eqref{vm2} and \eqref{vm4}.
\end{proof}

\section{Construction of the Test Function\label{s5}}

\subsection{Preliminaries}

Let $\bb{M}$ be the symmetric $K\times K$ matrix defined by
\begin{equation*}
\bb{M}_{ij}\;=\;\begin{cases}
\frac{1}{\mu}\sum_{l=1}^{K}\kappa_{i,l} & \mbox{if }i=j\\
-\frac{1}{\mu}\kappa_{i,j} & \mbox{if}\mbox{ }i\neq j
\end{cases}\;\;;1\le i,j\le K\;,
\end{equation*}
so that
\begin{equation}
\mc{D}(\mb{x})\;=\;\mb{x}^{T}\bb{M}\mb{x}\;.\label{d20}
\end{equation}
Define two subsets of $\bb{R}^{K}$ by
\begin{align*}
&\mb{N} \;=\;   \{\mb{x}\in\bb{R}^{K}:x_{1}=x_{2}=\cdots=x_{K}\}\;,\\
&\mb{R}   \;=\;   \{\mb{x}\in\bb{R}^{K}:x_{1}+x_{2}+\cdots+x_{K}=0\}\;.
\end{align*}

\begin{lem}
\label{lem51}The null-space and range of the matrix $\bb{M}$
are $\mb{N}$ and $\mb{R}$ respectively. \end{lem}
\begin{proof}
Suppose that $\bb{M}\mb{b}=\mb{0}$. Then, by \eqref{d20}
we have $\mc{D}(\mb{b})=0$. Hence, $\mb{b}\in\mb{N}$
as we observed in the line following \eqref{diri}. On the other hand,  any $\mb{b\in}\mb{N}$ satisfies $\bb{M}\mb{b}=\mb{0}$.
Hence, the null-space of $\bb{M}$ is $\mb{N}$.
Since the dimension of the null-space is $1$, that of the range of
$\bb{M}$ must be $(K-1)$ dimensional. Since $\sum_{i=1}^{K}(\bb{M}\mb{b})_{i}=0$ for all
$\mb{b}\in\bb{R}^{K}$, the range of $\bb{M}$ is a subset
of $\mb{R}$. Since $\dim(\mb{R})=K-1$, we can conclude that
$\mb{R}$ is the range of $\bb{M}$.
\end{proof}
For $\mb{c}\in\mb{R}$, write $$\bb{M}^{-1}\mb{c}=\{\mb{b}\in\bb{R}^{K}:\bb{M}\mb{b}=\mb{c}\}$$
Then, for $\mb{b}\in\bb{M}^{-1}\mb{c}$, we can write
$\bb{M}^{-1}\mb{c}=\mb{N}+\mb{b}$. Hence,
we can observe that $\mc{D}(\cdot)$ is a constant function on
$\bb{M}^{-1}\mb{c}$.

Now define a function $\mc{D}_{\mb{c}}:\bb{R}^{K}\rightarrow\bb{R}$
by
\begin{equation}
\mc{D}_{\mb{c}}(\mb{x})\;=\;\mc{D}(\mb{x})-2\,\mb{c}\cdot\mb{x}\;.\label{le52}
\end{equation} Then, by \eqref{d20}, for $\mb{b}\in\bb{M}^{-1}\mb{c}$,
\begin{equation}
\mc{D}_{\mb{c}}(\mb{b})\;=\;\mb{b}\cdot\bb{M}\mb{b}-2\,\mb{c}\cdot\mb{b}
\;=\;-\mb{b}\cdot\bb{M}\mb{b}
\;=\;-\mc{D}(\mb{b})\;,\label{le53}
\end{equation}
and hence the function $\mc{D}_{\mb{c}}(\cdot)$ restricted
to $\bb{M}^{-1}\mb{c}$ is a constant function as well. Let us denote that
constant by $\mc{D}_{\mb{c}}(\bb{M}^{-1}\mb{c})$, with slight abuse of notation.
\begin{lem}
\label{l51}Fix $\mb{c}\in\mb{R}$. Then, $\mc{D}_{\mb{c}}(\bb{M}^{-1}\mb{c})$
is the minimum of $\mc{D}_{\mb{c}}(\cdot)$.
Furthermore, if $\mc{D}_{\mb{c}}(\mb{x})\le\mc{D}_{\mb{c}}(\bb{M}^{-1}\mb{c})+\delta$
for some $\delta>0$, then there exists $\mb{b}_{0}\in\bb{M}^{-1}\mb{c}$
such that $|\mb{x}-\mb{b}_{0}|\le C\sqrt{\delta}$ for some
constant $C>0$ not depending on $\delta$. \end{lem}
\begin{proof}
Let $\mb{b}\in\bb{M}^{-1}\mb{c}$. Then, since $\bb{M}$ is symmetric, it is easy to observe that for any $\mb{x}$,
\begin{equation*}
\mc{D}_{\mb{c}}(\mb{x})\;=\;\mc{D}_{\mb{c}}(\mb{b})+\mc{D}(\mb{t})\;.
\end{equation*}
where $\mb{t} := \mb{x} - \mb{b}$.
The first part of the lemma follows since $\mc{D}$ is a non-negative function. As for the second part,
we must have $\mc{D}(\mb{t})\le\epsilon$, and hence, by continuity of $\mc{D}$ and the fact that the nullspace of $\mc{D}$ is $\mb{N}$, we can find $t\in\bb{R}$
such that $|t_{i}-t|\le C\sqrt{\delta}$ for all $1\le i\le K,$ where $t_{i}$ is such that $\mb{t} = (t_{1},\cdots,t_{K})$.
Then, $\mb{b}_{0}:=\mb{b}+(t,t,\cdots,t)\in\mb{b}+\mb{N}=\bb{M}^{-1}\mb{c}$
fulfills the requirement of the second part of the lemma.
\end{proof}

\subsection{Test function}

Denote by $\mb{\chi}_{\mc{A}}(\cdot)$ the indicator function
of the set $\mc{A}\subset\bb{R}^{d}$. We emphasize that the following construction of the test function $\psi^\epsilon$ is the main ingredient in the proof of Theorem \ref{t12}, and contains most of technical difficulties of the problem.
\begin{theorem}
\label{t41}Fix a non-zero vector $\mb{c}\in\mb{R}$ and $\mb{b}\in\bb{M}^{-1}\mb{c}$.
Then, for each $\epsilon>0$, there exists a function $\psi^{\epsilon}\in W_{\textup{loc}}^{2,p}(\bb{R}^{d})\cap L^\infty (\bb{R}^d)$
for all $p\in[1,\infty)$ that satisfies the equation
\begin{equation}
-\textup{div}\left(\frac{\sigma^{\epsilon}}{\tau_{\epsilon}}\,D\psi^{\epsilon}\right)
\;=\;\sum_{i=1}^{K}\frac{c_{i}}{|\mc{V}_{i}|}\,\chi_{\mc{V}_{i}}\;,\label{e410}
\end{equation}
and the uniform energy estimate
\begin{equation}
 \sup_{0<\epsilon<1} {\int}_{\bb{R}^{d}}\frac{\sigma^{\epsilon}}
 {\tau_{\epsilon}}\left|D\psi^{\epsilon}\right|^{2} \;<\; \infty  \label{e411}
\end{equation}
and finally,
\begin{equation}
\lim_{\epsilon\rightarrow0}\,\sup_{1\le i\le K}\,\sup_{\xi\in\mc{V}_{i}}\left|\psi^{\epsilon}(\xi)-b_{i}\right|\;=\;0\;.\label{e412}
\end{equation}
\end{theorem}
The proof of this theorem is divided into several lemmas. We start by simplifying the problem and by introducing relevant notions before starting these lemmas.

By linearity, it suffices to prove the theorem
for $\mb{c}=\mb{e}_{i}-\mb{e}_{j}$ for some $i\neq j$.
Therefore, without loss of generality, we assume that $\mb{c}=\mb{e}_{1}-\mb{e}_{2}=(1,-1,0,\,\cdots,0)$,
so that $c_{1}=1$, $c_{2}=-1$, $c_{i}=0$ for $i\ge3$.

For $\phi\in H_{\textrm{loc}}^{1}(\bb{R}^{d})$, define a functional $I$
by
\begin{equation}\label{i}
I[\phi]\;=\;\frac{1}{2}\int_{\bb{R}^{d}}\frac{\sigma^{\epsilon}}{\tau_{\epsilon}}
\left|D\phi\right|^{2} \,d\xi
\,-\,\frac{1}{|\mc{V}_{1}|}\int_{\mc{V}_{1}}\phi \, d\xi\,+\,\frac{1}{|\mc{V}_{2}|}\int_{\mc{V}_{2}}\phi \,d\xi\;,
\end{equation}
and let $\phi^{\epsilon}$ be a minimizer of $I[\phi]$ on $H_{\textrm{loc}}^{1}(\bb{R}^{d})$.
Then the Euler-Lagrange equation for $\phi^{\epsilon}$ is \eqref{e410}
for $\mb{c}=\mb{e}_{1}-\mb{e}_{2}\in\bb{R}^{K}$,
and moreover $\phi^{\epsilon}\in W_{\textrm{loc}}^{2,p}(\bb{R}^{d})$
for all $p\in[1,\infty)$.

For $1\le i\le K$, define
\begin{equation*}
\lambda_{\epsilon,i}\;=\;\frac{1}{|\mc{V}_{i}|}\int_{\mc{V}_{i}}\phi^{\epsilon}\,d\xi\;.
\end{equation*}
Since $I[\phi^{\epsilon}]=I[\phi^{\epsilon}+c]$ for all $c\in\bb{R}$,
we can assume without loss of generality that
\begin{equation*}
\lambda_{\epsilon,1}\;=\;-\lambda_{\epsilon,2}\;:=\;\lambda_{\epsilon}\;.
\end{equation*}
Note that $\lambda_{\epsilon}\ge0$ since otherwise we can replace $\phi^{\epsilon}$ with $-\phi^{\epsilon}$.
Let $\mu_{\epsilon,i}:=\sup_{\xi\in\mc{V}_{i}}\left|\phi^{\epsilon}(\xi)\right|$
and define
\begin{equation}
\mu_{\epsilon}\;:=\;\max\{\mu_{\epsilon,1},\,\mu_{\epsilon,2}\}\;.\label{c0}
\end{equation}
Then we can assume that
\begin{equation*}
\sup_{\xi\in\bb{R}^{d}}\left|\phi^{\epsilon}(\xi)\right|\;=\;\mu_{\epsilon}\;,
\end{equation*}
since otherwise, $\bar{\phi}^{\epsilon}=\Lambda(\phi^{\epsilon})$ gives a lower value of $I$,
where
\begin{equation*}
\Lambda(s)=\begin{cases}
-\mu_{\epsilon} & \text{if }s\in(-\infty,\,-\mu_{\epsilon})\\
s & \text{if }s\in[-\mu_{\epsilon},\,\mu_{\epsilon}]\\
\mu_{\epsilon} & \text{if }s\in(\mu_{\epsilon},\,\infty)
\end{cases}
\end{equation*}

With the simplification and notations above, we now start the proof of the Theorem \ref{t41}. The first step is the following lemma.
\begin{lem}\label{leg1}
We have that
\begin{equation*}
\int_{\bb{R}^{d}}\frac{\sigma^{\epsilon}}{\tau_{\epsilon}}\,
|D\phi^{\epsilon}|^{2}\,d\xi\;=\;2\,\lambda_{\epsilon}\;.
\end{equation*}
\end{lem}
\begin{proof}
First observe that, since $\phi^{\epsilon}$ is a minimizer, $I[\phi^{\epsilon}]\le I[0]$, and so
\begin{equation}
\int_{\bb{R}^{d}}\frac{\sigma^{\epsilon}}{\tau_{\epsilon}}\,|D\phi^{\epsilon}|^{2}\,
d\xi\;\le\;4\,\lambda_{\epsilon}\;.\label{c2}
\end{equation}
For $R>0$, let $\mc{B}_{R}:= \{\xi\in\bb{R}^{d}:|\xi|\le R\}$
and let $\zeta_{R}:\bb{R}^{d}\rightarrow [0,\,1]$ be a smooth cutoff function with compact support such that $\zeta_{R}\equiv1$ on $B_{R}$, and $|D\zeta_{R}|\le1$.
Let us select $R$ large enough so that $\mc{V}_{i}\subset\mc{B}_{R}$
for all $1\le i\le K$ and $\Phi(\xi)>H+1$ on $(\mc{B}_{R})^{c}$
Then, multiplying \eqref{e410} by $\zeta_{R}\phi^{\epsilon}$ and integrating by parts, we obtain
\begin{equation}
\int_{\bb{R}^{d}}\frac{\sigma^{\epsilon}}{\tau_{\epsilon}}\,|D\phi^{\epsilon}|^{2}
\,\zeta_{R}\,d\xi
\;=\;
2\,\lambda_{\epsilon}-\int_{\bb{R}^{d}}\frac{\sigma^{\epsilon}}{\tau_{\epsilon}}\,
\phi^{\epsilon}\,D\phi^{\epsilon}\cdot D\zeta_{R}\,d\xi\;.\label{c3}
\end{equation}
Because $|D\zeta_{R}|\le1$, the square of the last term is bounded by
\begin{equation}
\left( \int_{\bb{R}^{d}}\frac{\sigma^{\epsilon}}{\tau_{\epsilon}}
\left|D\phi^{\epsilon}\right|^{2}
\,d\xi\, \right) \left( \int_{(\mc{B}_{R})^{c}}\frac{\sigma^{\epsilon}}{\tau_{\epsilon}}\,d\xi\; \right).\label{c4}
\end{equation}
Note that by the assumption $\Phi(\xi)>H+1$ on $(\mc{B}_{R})^{c}$ and by Lemma \ref{lem32},
the last integral converges to $0$ as $R\rightarrow\infty$ Hence, by our priori bound \eqref{c2}, the expression in
\eqref{c4} vanishes as $R\rightarrow\infty$. Hence, the proof is completed
 by letting $R\rightarrow\infty$ in \eqref{c3}.
\end{proof}

Recall the definition of $\mc{V}_{i}$ from \eqref{e01}. Let
us take $0<\eta'<\eta$ and let $\widetilde{\mc{V}}_{i}$, $1\le i\le K$,
be the connected component of
\begin{equation*}
\{\xi\in\mc{W}_{i}:\Phi(\xi)\,<\,H-\eta'\}
\end{equation*}
containing $\mc{V}_{i}$. Then, we can obtain the following $L^2$-estimate for $\phi^\epsilon-{\lambda_\epsilon,i}$ on the extended valley $\widetilde{\mathcal{V}}_i$, for all $1\le i\le K$.
\begin{lem}\label{leg2}For all $1\le i\le K$, it holds that
\begin{equation*}
\left\Vert \phi^{\epsilon}-\lambda_{\epsilon,i}\right\Vert _{L^{2}(\widetilde{\mc{V}}_{i})}\;=\;o_\epsilon(1)\lambda_\epsilon^{\frac{1}{2}}
\;.
\end{equation*}
\end{lem}
\begin{proof}
Define
\begin{equation*}
\widetilde{\lambda}_{\epsilon,i}\;=\;\frac{1}{|\widetilde{\mc{V}}_{i}|}
\int_{\widetilde{\mc{V}}_{i}}\phi^{\epsilon}\,d\xi\;\;;\;1\le i\le K\;.
\end{equation*}
By using Poincar\'{e}'s inequality, as well as \eqref{e14} and Lemma \ref{leg1}, we get that for
all $1\le i\le K$,
\begin{multline}
 \int_{\widetilde{\mc{V}}_{i}}|\phi^{\epsilon}-\widetilde{\lambda}_{\epsilon,i}|^{2}
 \,d\xi \;  \le   \; C\int_{\widetilde{\mc{V}}_{i}}\left|D\phi^{\epsilon}\right|^{2}\,d\xi\;\le\; C\,\epsilon^{(d/2)-1}\,e^{-\frac{\eta'}{\epsilon}}
 \int_{\widetilde{\mc{V}}_{i}}\frac{\sigma^{\epsilon}}
 {\tau_{\epsilon}}\left|D\phi^{\epsilon}\right|^{2}\,d\xi  \\
  \le \; C\,\epsilon^{(d/2)-1}\,e^{-\frac{\eta'}{\epsilon}}\,\lambda_{\epsilon}\;.\label{c41}
\end{multline}
Hence, we can derive
\begin{multline}
\left|\lambda_{\epsilon,i}-\widetilde{\lambda}_{\epsilon,i}\right|
\;=\;
\left|\frac{1}{|\mc{V}_{i}|}\int_{\mc{V}_{i}}(\phi^{\epsilon}-
\widetilde{\lambda}_{\epsilon,i})\,d\xi\right|
\;\le\;
 C\left[\int_{\widetilde{\mc{V}}_{i}}|\phi^{\epsilon}-
\widetilde{\lambda}_{\epsilon,i}|^{2}\,d\xi\right]^{\frac{1}{2}}\\
=\;o_{\epsilon}(1)\,\lambda_{\epsilon}^{\frac{1}{2}}\;.\label{c5}
\end{multline}
Now, combining \eqref{c41} and \eqref{c5} completes the proof of lemma.
\end{proof}
The next step is to enhance the previous $L^2$-estimate on extended valley $\widetilde{\mc{V}}_i$ to the $L^\infty$-estimated on the original valley $\mc{V}_i$. The proof is based on the elliptic estimate on $\phi^\epsilon$, and on a bootstrapping argument. Let us fix $p\in(d,\,\infty)$ from now on, and regard $p$ just as a constant.
\begin{lem}\label{leg3}
For all $1\le i\le K$, it holds that,
\begin{equation*}
\left\Vert \phi^{\epsilon}-\lambda_{\epsilon,i}\right\Vert _{L^{\infty}(\mc{V}_{i})}\;\le \; o_{\epsilon}(1)\,\big(1+\lambda_{\epsilon}^{1-\frac{1}{p}}\big)\;.
\end{equation*}
\end{lem}
\begin{proof}We fix $1\le i\le K$. On the set $\mc{W}_{i}$, the function $\phi^{\epsilon}$
satisfies
\begin{equation*}
-\mbox{div}\left(\frac{\sigma^{\epsilon}}{\tau_{\epsilon}}D\phi^{\epsilon}\right)
\;=\;\frac{c_{i}}{|\mc{V}_{i}|}\chi_{\mc{V}_{i}}\;.
\end{equation*}
This equation can be rewritten as
\begin{equation*}
-\Delta(\phi^{\epsilon}- {\lambda}_{\epsilon,i})\;=\;
-\frac{1}{\epsilon}\textup{div}\left[(\phi^{\epsilon}- {\lambda}_{\epsilon,i})
\,D\Phi\right]\,+\,
\frac{1}{\epsilon}(\phi^{\epsilon}- {\lambda}_{\epsilon,i})
\,\Delta\Phi\,+\,\frac{\tau_{\epsilon}}{\sigma^{\epsilon}}\,\frac{c_{i}}{|\mc{V}_{i}|}
\,\chi_{\mc{V}_{i}}\;,
\end{equation*}
and, therefore, standard regularity estimates for elliptic PDE (cf. \cite[Theorem 8.17]{GT}) imply
\begin{equation*}
\left\Vert \phi^{\epsilon}- {\lambda}_{\epsilon,i}\right\Vert _{L^{\infty}(\mc{V}_{i})}   \;\le\;
C\left\Vert \phi^{\epsilon}- {\lambda}_{\epsilon,i}\right\Vert _{L^{2}(\widetilde{\mc{V}}_{i})}
\,+\,
\frac{C}{\epsilon}\left\Vert \phi^{\epsilon}- {\lambda}_{\epsilon,i}\right\Vert _{L^{p}(\widetilde{\mc{V}}_{i})}
\,+\,
C\left\Vert \frac{\tau_{\epsilon}}{\sigma^{\epsilon}}\right\Vert _{L^{\infty}(\mc{V}_{i})}\;,
\end{equation*}
since $p\in(d,\infty)$.
By Lemma \ref{leg2} and H\"{o}lder's inequality, (similar
argument to \cite[(3.31)]{ET}) we obtain
\begin{equation}
\left\Vert \phi^{\epsilon}-\lambda_{\epsilon,i}\right\Vert _{L^{\infty}(\mc{V}_{i})}
\;\le\; o_{\epsilon}(1)
\left[\,\lambda_{\epsilon}^{\frac{1}{2}}\,+\,C\left\Vert \phi^{\epsilon}-\lambda_{\epsilon,i}\right\Vert _{L^{\infty}(\widetilde{\mc{V}}_{i})}^{1-\frac{2}{p}}\lambda_{\epsilon}^{\frac{1}{p}}
\,+1\right]\;.\label{d1}
\end{equation}
Recall the definition of $\mu_{\epsilon}$ from \eqref{c0} and write
$\mu_{\epsilon}=\mu_{\epsilon,k}$ where $k$ is either $1$ or $2$. Then,
we obtain
\begin{multline}
      \left\Vert \phi^{\epsilon}-\lambda_{\epsilon,i}\right\Vert _{L^{\infty}(\widetilde{\mc{V}}_{i})}
      \;\le\;
      2\,\mu_{\epsilon}
      \;=\;2\,\mu_{\epsilon,k}
      \;\le\;2\left|\lambda_{\epsilon,k}\right|\,+\,2\,
      \left\Vert \phi^{\epsilon}-\lambda_{\epsilon,k}\right\Vert _{L^{\infty}(\mc{V}_{k})}
  \\
      =\;2\,\lambda_{\epsilon}+2\left\Vert \phi^{\epsilon}-\lambda_{\epsilon,k}\right\Vert _{L^{\infty}(\mc{V}_{k})}\;.\label{d2}
\end{multline}
By inserting this result into \eqref{d1} with $i=k$, we derive
\begin{equation*}
\left\Vert \phi^{\epsilon}-\lambda_{\epsilon,k}\right\Vert _{L^{\infty}(\mc{V}_{k})}
\;\le\;
 o_{\epsilon}(1)\left[1+2\left\Vert \phi^{\epsilon}-\lambda_{\epsilon,k}\right\Vert _{L^{\infty}(\mc{V}_{k})}^{1-\frac{2}{p}}\,\lambda_{\epsilon}^{\frac{1}{p}}
 \,+\,\lambda_{\epsilon}^{1-\frac{1}{p}}\,+\,
 \lambda_{\epsilon}^{\frac{1}{2}}\right]\,.
\end{equation*}
Therefore, by H\"{o}lder's inequality, we conclude that
\begin{equation}
\left\Vert \phi^{\epsilon}-\lambda_{\epsilon,k}\right\Vert _{L^{\infty}(\mc{V}_{k})}\;\le\; o_{\epsilon}(1)\,\big(1+\lambda_{\epsilon}^{1-\frac{1}{p}}\big)\;.\label{d3}
\end{equation}
By inserting \eqref{d3} into \eqref{d2}, we obtain,
\begin{equation}
\left\Vert \phi^{\epsilon}-\lambda_{\epsilon,i}\right\Vert _{L^{\infty}(\widetilde{\mc{V}}_{i})}
\;\le\;
2\,\lambda_{\epsilon}+
o_{\epsilon}(1)\,\big(1+\lambda_{\epsilon}^{1-\frac{1}{p}}\big)
 \;.\label{d4}
\end{equation}
Finally, the proof of lemma is completed by inserting this into \eqref{d1}.
\end{proof}
In view of the previous lemma, it is important to prove that $\lambda_\epsilon$ is bounded by a constant for small enough $\epsilon$. Indeed, we are able to establish more than this, as in the following lemma. The following lemma is the most renovative part of the current paper.
\begin{lem}\label{leg4}We have that,
\begin{equation*}
\lim_{\epsilon\rightarrow0}\lambda_{\epsilon}\;=\;
-\frac{1}{2}\,\mc{D}_{\mb{c}}(\mb{b})
\;=\;
\frac{1}{2}\,\mc{D}(\mb{b})\;.
\end{equation*}
\end{lem}
\begin{proof}
Recall $\mb{b}\in\bb{M}^{-1}\mb{c}$ from the statement
of theorem. Hence, by \eqref{le53}, the second identity of the lemma is obvious.

Now we focus on the first identity of the lemma. Let $\varphi_{\mb{b}}^{\epsilon}$ be the minimizer
of the variational problem on the left-hand-side of \eqref{var2}.
Since $I[\phi^{\epsilon}]=-\lambda_{\epsilon}$ by Lemma \ref{leg1} and
since $\phi^{\epsilon}$ is the minimizer of $I[\phi]$, by Theorem
\ref{th41}, we obtain
\begin{equation}
-\lambda_{\epsilon} \;=\; I[\phi^\epsilon]
 \;\le\;
 I[\varphi_{\mb{b}}^{\epsilon}]
\;  = \;  \frac{1}{2}\int_{\bb{R}^{d}}\frac{\sigma^{\epsilon}}{\tau_{\epsilon}}\,
|D\varphi_{\mb{b}}^{\epsilon}|^{2}\,d\xi
\,-\,
\sum_{i=1}^{K}c_{i}b_{i}\,
=\;\frac{1}{2}\,\mc{D}_{\mb{c}}(\mb{b})\,+\,o_{\epsilon}(1)\;.\label{con11}
\end{equation}
Therefore,  we get
\begin{equation}
\liminf_{\epsilon\rightarrow0}\lambda_{\epsilon}\;\ge\;
-\frac{1}{2}\,\mc{D}_{\mb{c}}(\mb{b})
\;=\;
\frac{1}{2}\,\mc{D}(\mb{b})\;>\;0\;,\label{con111}
\end{equation}
where the last inequality is strict since $\mb{c}\neq\mb{0}$. This proves the
half of the first identity.

We now have to prove the reverse inequality, namely,
\begin{equation}\label{con112}
\limsup_{\epsilon\rightarrow0} \lambda_{\epsilon} \;\le\; - \frac{1}{2} \mc{D}_{\mb{c}}(\mb{b})\;.
\end{equation}
This is the crux of the proof. Let
$\beta_{1}\le\beta_{2}\le\cdots\le\beta_{K}$
be the enumeration of the numbers $\frac{\lambda_{\epsilon,1}}{\lambda_{\epsilon}}$,
$\cdots$, $\frac{\lambda_{\epsilon,K}}{\lambda_{\epsilon}}$. Strictly speaking, $\beta_{i} = \beta_{i,\epsilon},$ but we will ignore the dependency on $\epsilon$ for the time being. Fix $\delta>0$
and $L>2$. We now introduce an auxiliary function $\Gamma(x)$. We refer to Figure \ref{fig2} for the visualization of the construction below. For $1\le i\le j\le K$, we say that $B_{i,j}=\{\beta_{k}:i\le k\le j\}$
is a {\it{good set}} if
\begin{equation}
\begin{cases}
\beta_{k+1}-\beta_{k}\le L^{k}\delta & \mbox{for }i\le k\le j-1\;,\\
\beta_{k+1}-\beta_{k}>L^{k}\delta & \mbox{for }k=i-1\mbox{ and }k=j\;,
\end{cases}\label{go1}
\end{equation}
where $\beta_{0}:=-\infty$ and $\beta_{K+1}:=\infty$. Enumerate all good
sets by
\begin{equation*}
B_{j_{0}+1,j_{1}},\;B_{j_{1}+1,j_{2}},\,\cdots,\,B_{j_{M-1}+1,j_{M}}\;,
\end{equation*}
where $j_{0}:=0$ and $j_{M}:=K$. For $1\le k\le M$, define
\begin{equation*}
I_{k}\;=\;[\,\beta_{j_{k-1}+1}-\delta,\;\beta_{j_{k}}+\delta,]\;.
\end{equation*}

\begin{figure}
  \protect
\includegraphics[scale=0.200]{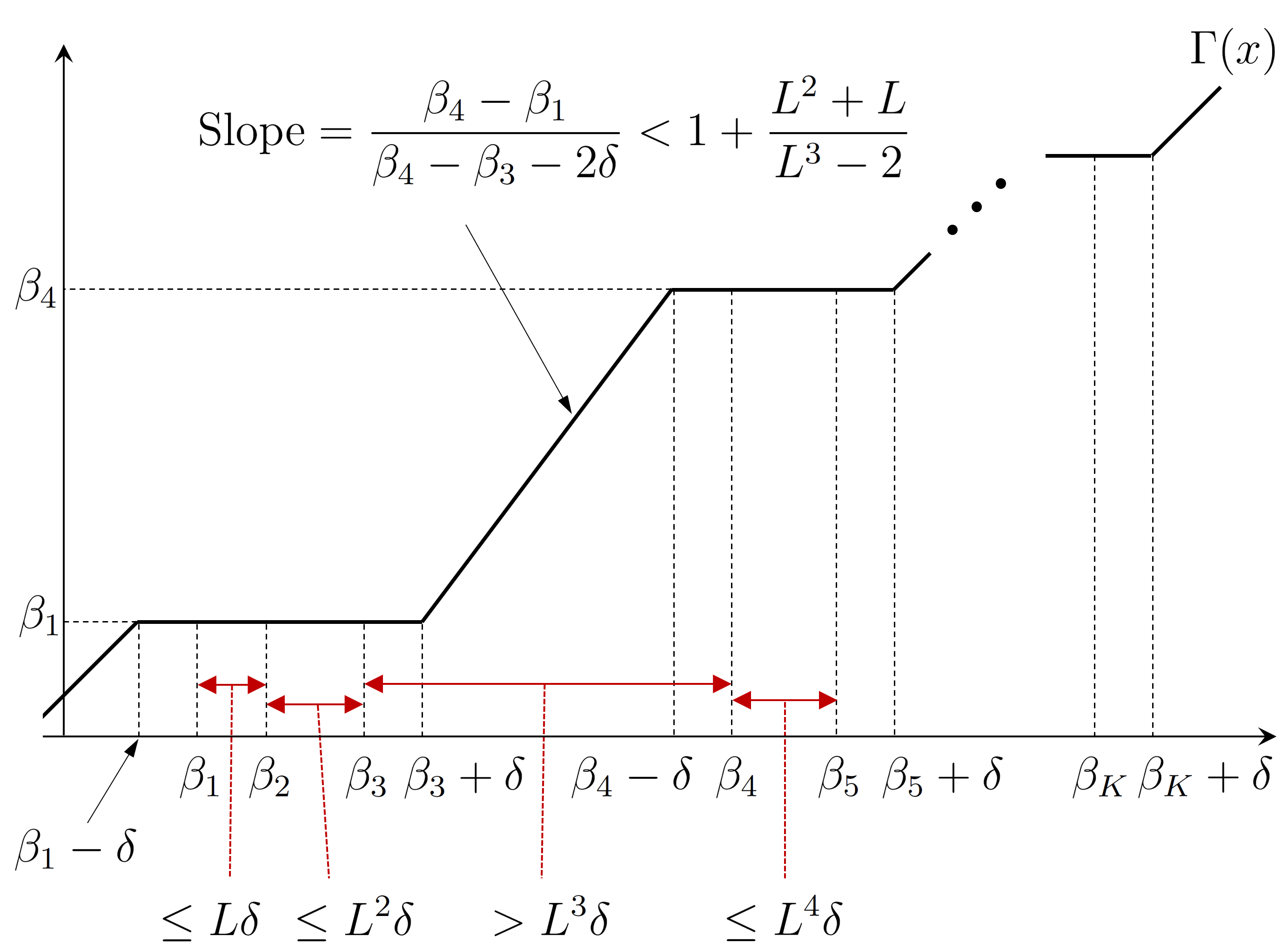}\protect
  \caption{\label{fig2}Illustration of the piecewise linear function $\Gamma(x)$: In this example, $B_{1,3}$ and $B_{4,5}$ are good sets, since $(\beta_1, \beta_2, \beta_3)$ and $(\beta_4, \beta_5)$ form cluster in the sense defined in \eqref{go1}, and accordingly, we have $j_1 =3$ and $j_2 =5$. In \eqref{go3} we show that the slope of each piece is bounded by $1+o_L(1)$.}
\end{figure}

We now define a piecewise linear function $\Gamma=\Gamma^{\epsilon,\delta,L}:\bb{R}\rightarrow\bb{R}$
by:
\begin{equation*}
\Gamma(x)=\begin{cases}
x+\delta & \mbox{if }x\in(-\infty,\,\beta_{1}-\delta)\\
\beta_{j_{k-1}+1} & \mbox{if }x\in I_{k},\;1\le k\le M\\
\beta_{j_{k-1}+1}+\dfrac{\beta_{j_{k}+1}-\beta_{j_{k-1}+1}}
{\beta_{j_{k}+1}-\beta_{j_{k}}-2\delta}(x-(\beta_{j_{k}}+\delta)) &\parbox{15em}{if $x\in(\beta_{j_{k}}+\delta,\beta_{j_{k}+1}-\delta)$,\\ \mbox{ }  $\,1\le k\le M-1$} \\
x-(\beta_{K}+\delta)+\beta_{j_{M-1}+1} & \mbox{if }x\in(\beta_{K}+\delta,\,\infty)\;.
\end{cases}
\end{equation*}
By construction, the function $\Gamma$ is continuous. Now we estimate the slope of $\Gamma$ on the interval $(\beta_{j_{k}}+\delta,\beta_{j_{k}+1}-\delta)$.  By \eqref{go1},
\begin{multline*}
\frac{\beta_{j_{k}+1}-\beta_{j_{k-1}+1}}{\beta_{j_{k}+1}-\beta_{j_{k}}-2\delta}   \;=\;   \frac{\beta_{j_{k}+1}-\beta_{j_{k}}+\sum_{i=j_{k-1}+1}^{j_{k}-1}
(\beta_{i+1}-\beta_{i})}{\beta_{j_{k}+1}-\beta_{j_{k}}-2\delta}\\
  \;<\;   \frac{L^{j_{k}}\delta+\sum_{i=j_{k-1}+1}^{j_{k}-1}L^{i}\delta}
  {L^{j_{k}}\delta-2\delta}
  \;<\;  \frac{L^{j_{k}}+KL^{j_{k}-1}}{L^{j_{k}}-2} \,.
\end{multline*}
Hence, since $L > 2$, we conclude that
\begin{equation}
\Vert\Gamma'\Vert_\infty  \;\le\; 1+\frac{K+2}{L-2}  \;=\; 1+o_L (1)\;,\label{go3}
\end{equation}
where $o_{L}(1)$ is a term vanishing
when $L\rightarrow\infty$ and independent of $\epsilon$ and $\delta$. Finally, observe that $\Gamma$ is constant on intervals of the form $[(\lambda_{\epsilon,i}/\lambda_{\epsilon})-\delta,\,(\lambda_{\epsilon,i}/\lambda_{\epsilon})+\delta]$,
$1\le i\le K$, and denote that constant by $\gamma_{i}=\gamma_{\epsilon,\delta,L,i}.$
Then, we obtain,
\begin{equation}
\left|\gamma_{i}-\frac{\lambda_{\epsilon,i}}{\lambda_{\epsilon}}\right|
\;\le\;
\max_{1\le k\le M}(\beta_{j_{k}}-\beta_{j_{k-1}+1})
\;\le\;
\max_{1\le k\le M} \sum_{i=j_{k-1}+1}^{j_{k}-1}L^{i}\delta \;\le\; KL^{K}\delta\;.\label{go4}
\end{equation}
Define
$$\mb{g}\;:=\;(\gamma_{1},\gamma_{2},\cdots,\gamma_{K})\;.$$
Then, since $\mb{g}\cdot\mb{c}=\gamma_{1}-\gamma_{2}$ and
since $\lambda_{\epsilon,1}=-\lambda_{\epsilon,2}=\lambda_{\epsilon}$,
we deduce from \eqref{go4} that
\begin{equation}
|\,\mb{g}\cdot\mb{c}-2\,|\;\le\;2KL^{K}\delta\;.\label{go5}
\end{equation}
Now, define
\begin{equation*}
\widehat{\phi}\;=\;\widehat{\phi}^{\epsilon,\delta,L}\;:=\;
\Gamma\left(\frac{\phi^{\epsilon}}{\lambda_{\epsilon}}\right).
\end{equation*}
Then, in view of  Lemma \ref{leg3}, \eqref{con111}, and \eqref{go4}, we have that $\widehat{\phi}\in\mc{F}_{\mb{g}}$
(cf. \eqref{fb}) for all sufficiently small $\epsilon$. Assume from now on that $\epsilon$ is small enough so that this condition is valid. (Hence, we should send $\epsilon\to 0$ before taking any limit for $\delta$ or $L$.) Then, by
Theorem \ref{th41} and \eqref{go5},
\begin{multline*}
\frac{1}{2}\int_{\bb{R}^{d}}\frac{\sigma^{\epsilon}}
{\tau_{\epsilon}}\left|D\widehat{\phi}\,\right|^{2} \, d\xi \; \ge   \; \frac{1+o_{\epsilon}(1)}{2}\, \mc{D}(\mb{g}) \;=\;
\frac{1+o_{\epsilon}(1)}{2\lambda_{\epsilon}^{2}}\,\mc{D}
(\lambda_{\epsilon}\mb{g})  \\
  =  \; \frac{1+o_{\epsilon}(1)}{\lambda_{\epsilon}^{2}}
  \left[\frac{1}{2}\,\mc{D}_{\mb{c}}(\lambda_{\epsilon}\,\mb{g})+
  \lambda_{\epsilon}\,(\mb{g}\cdot\mb{c})\right]\\
   \ge\;
   \frac{1+o_{\epsilon}(1)}{\lambda_{\epsilon}^{2}}
   \left[\frac{1}{2}\,\mc{D}_{\mb{c}}(\lambda_{\epsilon}\mb{g})\,+\,
   \lambda_{\epsilon}\,(2-2KL^{K}\delta)\right]  .\label{con2}
\end{multline*}
On the other hand, by \eqref{go3} and by Lemma \ref{leg1},
\begin{equation*}
\frac{1}{2}\,\int_{\bb{R}^{d}}\frac{\sigma^{\epsilon}}{\tau_{\epsilon}}
\left|D\widehat{\phi}\,\right|^{2}\,d\xi
\;\le\;
\frac{1+o_{L}(1)}{2\lambda_{\epsilon}^{2}}\
\int_{\bb{R}^{d}}\frac{\sigma^{\epsilon}}{\tau_{\epsilon}}
\left|D\phi^{\epsilon}\right|^{2}\,d\xi
\;=\;\frac{1+o_{L}(1)}{\lambda_{\epsilon}}\;.\label{con3}
\end{equation*}
Combining those two inequalities, we obtain
\begin{equation*}
\frac{1+o_{L}(1)}{\lambda_{\epsilon}}
\;\ge\;
\frac{1+o_{\epsilon}(1)}{\lambda_{\epsilon}^{2}}\left[\frac{1}{2}\,\mc{D}_{\mb{c}}
(\lambda_{\epsilon}\mb{g})+\lambda_{\epsilon}\,(2-2KL^{K}\delta)\right].
\end{equation*}
We select $\epsilon$ small enough so that the $o_\epsilon(1)$ term is greater than $-1$. Then, we can re-organize the previous inequality as
 \begin{equation}
\left[(1-2KL^{K}\delta) +o_\epsilon(1)+ o_{L}(1)\right]
\lambda_{\epsilon}
\;\le \;-\frac{1 }{2}\,\mc{D}_{\mb{c}}
(\lambda_{\epsilon}\mb{g})
\;\le \;-\frac{1 }{2}\,\mc{D}_{\mb{c}}(\mb{b})
\;,\label{con31}
\end{equation}
since $\mb{b}=\bb{M}^{-1}\mb{c}$ is the minimizer of $\mc{D}_\mb{c}$ by Lemma \ref{l51}.
Now take $L$ large enough so that $o_{L}(1)<1/3$ and then take $\delta$ small
enough so that $2KL^{K}\delta<1/3$. In this way, the quantity in brackets converges to a positive number as $\epsilon\rightarrow 0$. Then, by taking $\limsup_{\epsilon\rightarrow 0}$ in the previous inequality, we obtain
\begin{equation*}
\left[1-2KL^{K}\delta-o_{L}(1)\right]\,\limsup_{\epsilon\rightarrow0}
\lambda_{\epsilon}
\;\le\;
-\frac{1}{2}\,\mc{D}_{\mb{c}}(\mb{b})\;.
\end{equation*}
Finally, send $\delta\rightarrow0$ and then $L\rightarrow\infty$
to conclude \eqref{con112}. This finished the proof.
\end{proof}

As a direct consequence of the previous lemma, we obtain the following boundedness results.
\begin{cor}\label{cor1}There exist constants $C>0$ and $\epsilon_0>0$ such that, for all $\epsilon\in(0,\,\epsilon_0)$,
\begin{equation*}
\lambda_\epsilon \le C\;\;\;\mbox{and}\;\;\;|\lambda_{\epsilon,i}|\le C \;;\;\;1\le i\le K\;.
\end{equation*}
\end{cor}
\begin{proof}
The first inequality is immediate from Lemma \ref{leg4}. By this inequality and Lemma \ref{leg3}, we can conclude that  $|\mu_{\epsilon}-\lambda_{\epsilon}|=o_{\epsilon}(1)$. This implies the second inequality of the corollary since
 $|\lambda_{\epsilon,i}|\le \mu_\epsilon \le \lambda_{\epsilon}+o_{\epsilon}(1)$.
\end{proof}
Now we arrive the last ingredient for the proof of Theorem \ref{t41}
\begin{lem}\label{leg5}Define  $\mb{l}_{\epsilon}=(\lambda_{\epsilon,1},\lambda_{\epsilon,2},
\cdots,\lambda_{\epsilon,K})$. Then, we have that,
\begin{equation*}
\lim_{\epsilon\rightarrow0}\mc{D}_{\mb{c}}(\mb{l}_{\epsilon})
\;=\;\mc{D}_{\mb{c}}(\mb{b})\;.
\end{equation*}
\end{lem}
\begin{proof}
By \eqref{go4} and the boundedness of $\lambda_{\epsilon,i}$ obtained in the previous corollary, we have
\begin{equation}
\left|\mc{D}_{\mb{c}}(\lambda_{\epsilon}\mb{g})\,-\,\mc{D}_{\mb{c}}(\mb{l}_{\epsilon})
\right|
\;\le\;
 C(L^{2K}\delta^{2}\,+\,L^{K}\delta)\label{con32}
\end{equation}
for some constant $C>0$. Combining this bound and the first inequality of \eqref{con31} yields
\begin{equation*}
  \mc{D}_{\mb{c}}(\mb{l}_{\epsilon})\;\le \;
  -2\left[1-2KL^{K}\delta+o_{L}(1)
+o_\epsilon(1)\right]\lambda_\epsilon\,+\,C(L^{2K}\delta^{2}+L^{K}\delta)\;.
\end{equation*}
Thus, by Lemma \ref{leg4}
\begin{equation*}
\limsup_{\epsilon\rightarrow0}\mc{D}_{\mb{c}}(\mb{l}_{\epsilon})
\;\le\;
\left[1-2KL^{K}\delta-o_{L}(1)\right]\mc{D}_{\mb{c}}(\mb{b})
\,+\,C(L^{2K}\delta^{2}+L^{K}\delta)
\end{equation*}
By letting $\delta\rightarrow0$ and then $L\rightarrow\infty$, we
deduce
\begin{equation}
\limsup_{\epsilon\rightarrow0}\mc{D}_{\mb{c}}(\mb{l}_{\epsilon})
\;\le\;\mc{D}_{\mb{c}}(\mb{b})\;.\label{con34}
\end{equation}
On the other hand, we know by Lemma \ref{l51} that $\mc{D}_{\mb{c}}(\mb{l}_{\epsilon})\ge\mc{D}_{\mb{c}}(\mb{b})$
and thus,
\begin{equation}
\liminf_{\epsilon\rightarrow0}\mc{D}_{\mb{c}}(\mb{l}_{\epsilon})
\;\ge\;\mc{D}_{\mb{c}}(\mb{b})\;.\label{con35}
\end{equation}
Hence, by \eqref{con34} and \eqref{con35}, we can finish the proof of lemma.
\end{proof}
Now we arrived at the final stage of the proof.
\begin{proof}[Proof of Theorem \ref{t41}]
By Lemma \ref{leg5}, we can write $\mc{D}_{\mb{c}}(\mb{l}_{\epsilon})=\mc{D}_{\mb{c}}(\mb{b})+o_{\epsilon}(1)$.
By the second part of Lemma \ref{l51}, there exists $t_{\epsilon}\in\bb{R}$ such
that
\begin{equation}\label{cms1}
  \left|(\lambda_{i,\epsilon}-t_{\epsilon})-b_{i}\right|\;=\;o_{\epsilon}(1)\;\;\mbox{for all}\;1\le i\le K\;.
\end{equation}
We now define $\psi^{\epsilon}:=\phi^{\epsilon}-t_{\epsilon}$, and claim that $\psi^{\epsilon}$ satisfies all the requirements of the theorem.

First, the condition \eqref{e410} holds for $\phi^\epsilon$ since this function is chosen as a minimizer if the functional $I$ defined in \eqref{i}. Hence, the condition  \eqref{e410} is also valid for $\psi^{\epsilon}$ as well since $t_\epsilon$ is merely a constant so that $D\phi^\epsilon=D\psi^\epsilon$. Second, we can deduce \eqref{e411}  for $\phi^\epsilon$ by combining Lemma \ref{leg1} and Corollary \ref{cor1}. By the same reason as above, the condition \eqref{e411} also holds for $\psi^\epsilon$ as well.
Finally, the condition \eqref{e412} is immediately follows from \eqref{cms1}.
\end{proof}

\section{\label{s6}Proof of Theorem \ref{t12}}

The proof of Theorem \ref{t12} is similar to the proof of \cite[Theorem 3.7]{ET} and we present it below for sake of completeness.

\begin{proof}[Proof of Theorem \ref{t12}]
Define $\mc{A}_{1}=\{\xi:\Phi(\xi)\ge H+1\}$ and $\mc{A}_{2}=\{\xi:\Phi(\xi)\ge H+2\}$.
Let $\zeta:\bb{R}^d \rightarrow [0,\,1]$ be a smooth cutoff function
such that $\zeta\equiv1$ on $\mc{A}_{1}$ and $\zeta\equiv0$
on $(\mc{A}_{2})^{c}$.

Fix $\mb{b}\in\bb{R}^{K}\setminus\mb{N}$ and let $\mb{c}=\bb{M}\mb{b}(\neq\mb{0})$.
Then, denote by $\psi^{\epsilon}$ the function
in Theorem \ref{t41} with $\mb{c}\in\mb{R}$ and $\mb{b}\in\bb{M}^{-1}\mb{c}$
Let $f=f(x,t)\in C^{\infty}(U_{T})$ be a smooth test function.
Multiplying \eqref{p02} by $\zeta f\psi^{\epsilon}$ and integrating by parts, we obtain
\begin{multline}
 \int_{0}^{T}\int_{U}\int_{\mc{A}_{2}}\zeta f\psi^{\epsilon}\rho_{t}^{\epsilon}\,d\xi dxdt\,+\,
 \int_{0}^{T}\int_{U}\int_{\mc{A}_{2}}\psi^{\epsilon}\zeta a\,D_{x}f\cdot D_{x}\rho^{\epsilon}\,d\xi dxdt  \\
 =\;-\int_{0}^{T}\int_{U}\int_{\mc{A}_{2}}\frac{\sigma^{\epsilon}}{\tau{}_{\epsilon}}fD_{\xi}u^{\epsilon}\cdot D_{\xi}(\zeta\psi^{\epsilon})\,d\xi dxdt\;.\label{e50}
\end{multline}
Now we consider three integrals in \eqref{e50} separately.

Write $\mc{A}=\mc{V}_{1}\cup\cdots\cup\mc{V}_{K}$. Then, the first integral of \eqref{e50} can be split into
\begin{equation}
\int_{0}^{T}\int_{U}\int_{\mc{A}}\zeta f\psi^{\epsilon}\rho_{t}^{\epsilon}\,d\xi dxdt\,+\,
\int_{0}^{T}\int_{U}\int_{\mc{A}_{2}\mc{\setminus\mc{A}}}\zeta f\psi^{\epsilon}\rho_{t}^{\epsilon}\,d\xi dxdt\;.\label{e52}
\end{equation}
Since $\zeta\equiv1$ on $\mc{A}$, by \eqref{ee3} and by \eqref{e412}, we get
\begin{multline}
\int_{0}^{T}\int_{U}\int_{\mc{A}}\zeta f\psi^{\epsilon}\rho_{t}^{\epsilon}\,d\xi dxdt  \; = \;  [1+o_{\epsilon}(1)]\,\sum_{i=1}^{K}\int_{0}^{T}\int_{U}
\int_{\mc{V}_{i}}fb_{i}\rho_{t}^{\epsilon}\,d\xi dx dt  \\
   \longrightarrow   \; \int_{0}^{T}\int_{U}f\sum_{i=1}^{K}b_{i}\partial_{t}\alpha_{i}\,\,dxdt\;.
   \label{e53}
\end{multline}
The second term of \eqref{e52} becomes negligible since, by \eqref{ee32},
\begin{equation}
\left|\int_{0}^{T}\int_{\mc{A}_{2}\setminus\mc{A}}
\zeta f\psi^{\epsilon}\rho_{t}^{\epsilon}\,d\xi dt\right| \;\le\; C\int_{0}^{T}\int_{\Delta}|\rho_{t}^{\epsilon}|\,d\xi dxdt\rightarrow0.\label{e531}
\end{equation}

Now we consider the second integral of \eqref{e50}. Similarly, we split it into an integral on $\mc{A}$ and $\mc{A}_{2}\setminus\mc{A}$ respectively.
Then, by \eqref{ee4}, it is easy to verify that the integral on $\mc{A}_{2}\setminus\mc{A}$
vanishes as $\epsilon\rightarrow0$, while by \eqref{ee31} the integral
on $\mc{A}$ converges to
\begin{equation}
\int_{0}^{T}\int_{U}\,D_{x}f\cdot\sum_{i=1}^{K}a_{i}b_{i}D_{x}\alpha_{i}\,dxdt
\;=\;-\int_{0}^{T}\int_{U}f\cdot\sum_{i=1}^{K}a_{i}b_{i}\,
\Delta_{x}\alpha_{i}\,dxdt\label{e54}
\end{equation}
as $\epsilon\rightarrow0$.

Finally, integrating by parts again, the last term in \eqref{e50} becomes
\begin{multline}
\int_{0}^{T}\int_{U}\int_{\mc{A}_{2}}\frac{\sigma^{\epsilon}}
{\tau{}_{\epsilon}}f\left[-\psi^{\epsilon}D_{\xi}u^{\epsilon}+u^{\epsilon}
D_{\xi}\psi^{\epsilon}\right]\cdot D_{\xi}\zeta \,d\xi dxdt  \\
\quad+\int_{0}^{T}\int_{U}\int_{\mc{A}_{2}}f\zeta u^{\epsilon}\mbox{div}_{\xi}\left[\frac{\sigma^{\epsilon}}
{\tau{}_{\epsilon}}D_{\xi}\psi^{\epsilon}\right]\,d\xi dxdt\;.\label{e60}
\end{multline}
We claim that the first term is negligible. To this end, since
$D_{\xi}\zeta\equiv0$ on $\mc{A}_{1}$, by \eqref{ub}, by \eqref{e411}, the square of this integral is
bounded by
\begin{equation*}
C\int_{0}^{T}\int_{U}\int_{\mc{A}_{1}^{c}}\frac{\sigma^{\epsilon}}
{\tau{}_{\epsilon}}\, d\xi  dxdt\,\int_{0}^{T}\int_{U}\int_{\mc{A}_{1}^{c}}
\frac{\sigma^{\epsilon}}{\tau{}_{\epsilon}}
\left[\left|D_{\xi}u^{\epsilon}\right|^{2}
+\left|D_{\xi}\psi^{\epsilon}\right|^{2}\right]\,d\xi dxdt\;.
\end{equation*}
In the last expression, the first integral vanishes as $\epsilon\rightarrow0$
by Lemma \ref{lem32}, and the second integral is bounded because
of \eqref{uc} and \eqref{e411}. This proves the claim. On the other
hand, by Theorem \ref{t41} and \eqref{ee5}, the second integral
of \eqref{e60} converges to
\begin{equation}
-\int_{0}^{T}\int_{U}\sum_{i=1}^{K}fc_{i}
\frac{\alpha_{i}}{\widehat{\mu}_{i}}\,dxdt\;=\;
\int_{0}^{T}\int_{U}f\sum_{i,j=1}^{K}\alpha_{i}r_{i,j}(b_{j}-b_{i})
\,dxdt\;.\label{e61}
\end{equation}

By combining \eqref{e50}, \eqref{e53}, \eqref{e54} and \eqref{e61},
we obtain
\begin{equation*}
\int_{0}^{T}\int_{U}f\sum_{i=1}^{K}b_{i}
\left(\partial_{t}\alpha_{i}-a_{i}\Delta_{x}\alpha_{i}\right)\,dxdt\;=\;
\int_{0}^{T}\int_{U}f\sum_{i,j=1}^{K}\alpha_{i}r_{i,j}(b_{j}-b_{i})\,dxdt\;.
\end{equation*}
Now we select $\mb{b}=\mb{e}_{i}$ for $1\le i\le K$. Then,
the previous identity implies that
\begin{equation*}
\partial_{t}\alpha_{i}-a_{i}\Delta_{x}\alpha_{i}=\sum_{j=1}^{K}(r_{j,i}\alpha_{j}-r_{i,j}\alpha_{i})\;.
\end{equation*}
This completes the proof.
\end{proof}

\section*{Appendix: Proof of Lemmas \ref{lem21} and \ref{lem22}}
\begin{proof}[Proof of Lemma \ref{lem21}]
First of all, \eqref{ub} follows from the assumption $0 \leq u^{\epsilon}_{0} \leq C$ and from the maximum principle applied to \eqref{p02}

As for the energy estimate \eqref{uc}, multiplying \eqref{p02} by $u^{\epsilon}$ and integrating over $[0,t] \times \mathbb{R}^{d} \times U$ (where $t$ is fixed), we get
\begin{multline*}
\int_{0}^{t} \int_{\mathbb{R}^{d}} \int_{U} \sigma^{\epsilon}\, u^{\epsilon}_{t} \, u^{\epsilon} \,d\xi dx dt - \int_{0}^{t} \int_{\mathbb{R}^{d}} \int_{U} \sigma^{\epsilon}\, a \,(\Delta_{x} u^{\epsilon})\, u^{\epsilon} \,d\xi dx dt\\
= \;\frac{1}{\tau_{\epsilon}} \int_{0}^{t} \int_{\mathbb{R}^{d}} \int_{U} \textup{div}_{\xi}[\sigma^{\epsilon}\, D_{\xi} u^{\epsilon}]\, u^{\epsilon} \,d\xi dx dt\;.
\end{multline*}
Using $\sigma^{\epsilon} u^{\epsilon}_{t} u^{\epsilon} = \frac{1}{2} \sigma^{\epsilon} \partial_{t} \left| u^{\epsilon} \right|$, the first term becomes
\begin{multline*}
\int_{0}^{t} \int_{\mathbb{R}^{d}} \int_{U} \sigma^{\epsilon}\, u^{\epsilon}_{t} \, u^{\epsilon}\, d\xi dx dt \\
=\; \frac{1}{2} \int_{\mathbb{R}^{d}} \int_{U} \sigma^{\epsilon}  \left|u^{\epsilon}(x,\xi,t)\right|^{2} \,d\xi dx - \frac{1}{2} \int_{\mathbb{R}^{d}} \int_{U} \sigma^{\epsilon} \left|u^{\epsilon}_{0}\right|^{2} \,d\xi dx\;.
\end{multline*}
Applying the divergence theorem with respect to $x$ (note that there are no boundary terms since $\frac{\partial u^{\epsilon}}{\partial \nu} = 0$ on $\partial U \times \mathbb{R}^{d} \times (0,t)$), the second term becomes
\begin{equation*}
- \int_{0}^{t} \int_{\mathbb{R}^{d}} \int_{U} \sigma^{\epsilon}\, a \,(\Delta_{x} u^{\epsilon})\, u^{\epsilon}\, d\xi dx dt = \int_{0}^{t} \int_{\mathbb{R}^{d}} \int_{U} \sigma^{\epsilon}\, a \left| D_{x} u^{\epsilon} \right|^{2}\, d\xi dx dt\;.
\end{equation*}
Lastly, integrating by parts with respect to $\xi$ (there are again no boundary terms), the term on the right becomes
\begin{equation*}
\frac{1}{\tau_{\epsilon}} \int_{0}^{t} \int_{\mathbb{R}^{d}} \int_{U} \textup{div}_{\xi}[\sigma^{\epsilon}\, D_{\xi} u^{\epsilon}] \,u^{\epsilon}\, d\xi dx dt \;=\;
- \frac{1}{\tau_{\epsilon}} \int_{0}^{t} \int_{\mathbb{R}^{d}} \int_{U} \sigma^{\epsilon} \,\left| D_{\xi} u^{\epsilon} \right|^{2} \,d\xi dx dt\;.
\end{equation*}
Putting everything together, we obtain
\begin{multline*}
\frac{1}{2} \int_{\mathbb{R}^{d}} \int_{U} \sigma^{\epsilon} \left|u^{\epsilon}(x,\xi,t)\right|^{2}\, d\xi dx\, + \,\int_{0}^{t} \int_{\mathbb{R}^{d}} \int_{U} \sigma^{\epsilon} \left( a \left| D_{x} u^{\epsilon} \right|^{2} + \frac{1}{\tau_{\epsilon}} \left|D_{\xi} u^{\epsilon}\right| \right) \,d\xi dx dt \\=\;  \frac{1}{2} \int_{\mathbb{R}^{d}} \int_{U} \sigma^{\epsilon} \left|u^{\epsilon}_{0}\right|^{2}\, d\xi dx\,.
\end{multline*}
By our assumption \eqref{init} on the initial data $u^{\epsilon}_{0}$, the right-hand-side is bounded, and therefore, taking the supremum over $t \in [0,T]$, we obtain
\begin{equation*}
\sup_{0 \leq t \leq T}  \int_{\mathbb{R}^{d}} \int_{U} \sigma^{\epsilon} \left|u^{\epsilon}\right|^{2}\, d\xi dx
+ \int_{0}^{T} \int_{\mathbb{R}^{d}} \int_{U} \sigma^{\epsilon} \left( a \left| D_{x} u^{\epsilon} \right|^{2} + \frac{1}{\tau_{\epsilon}} \left| D_{\xi} u^{\epsilon} \right|^{2} \right) d\xi dx dt \leq C\,.
\end{equation*}
This gives us one part of our desired estimate; to obtain the other part, multiply \eqref{p02} by $u^{\epsilon}_{t}$ and integrate to obtain
\begin{multline}
\int_{0}^{t} \int_{\mathbb{R}^{d}} \int_{U} \sigma^{\epsilon} \left| u^{\epsilon}_{t} \right|^{2} \,d\xi dx dt - \int_{0}^{t} \int_{\mathbb{R}^{d}} \int_{U} a \,\sigma^{\epsilon} \left( \Delta_{x} u^{\epsilon} \right) u^{\epsilon}_{t} \,d\xi dx dt \\ = \frac{1}{\tau_{\epsilon}} \int_{0}^{t} \int_{\mathbb{R}^{d}} \int_{U} \textup{div}_{\xi} \left[ \sigma^{\epsilon} \, D_{\xi} u^{\epsilon} \right] u^{\epsilon}_{t} \,d\xi dx dt\;.\label{cm2}
\end{multline}
The first term stays as it is; as for the second integral, integrating by parts with respect to $x$ and using $D_{x} u^{\epsilon} \cdot D_{x} u^{\epsilon}_{t} = \frac{1}{2} \partial_{t} \left| D_{x} u^{\epsilon} \right|^{2}$, we can deduce that it equals to
\begin{equation*}
\frac{1}{2}\int_{\mathbb{R}^{d}} \int_{U} a\,\sigma^{\epsilon}  \left| D_{x} u^{\epsilon}(x,\xi,t) \right|^{2} \,dx d\xi \,- \,\frac{1}{2} \int_{\mathbb{R}^{d}} \int_{U} a\,\sigma^{\epsilon}   \left| D_{x} u_{0}^{\epsilon} \right|^{2} \,dx d\xi\;.
\end{equation*}
Similarly, for the last term of \eqref{cm2}, integrating by parts with respect to $\xi$, we can rewrite it as
\begin{equation*}
- \frac{1}{2\tau_{\epsilon}} \int_{\mathbb{R}^{d}} \int_{U} \sigma^{\epsilon} \left| D_{\xi} u^{\epsilon}(x,\xi,t) \right|^{2} \,dx d\xi \,+ \, \frac{1}{2\tau_{\epsilon}} \int_{\mathbb{R}^{d}} \int_{U} \sigma^{\epsilon} \left| D_{\xi} u_{0}^{\epsilon} \right|^{2} \,dx d\xi\;.
\end{equation*}
Putting everything together, we get
\begin{multline*}
\int_{0}^{t} \int_{\mathbb{R}^{d}} \int_{U} \sigma^{\epsilon} \left| u^{\epsilon}_{t} \right|^{2}\, d\xi dx dt \\
+ \frac{1}{2} \int_{\mathbb{R}^{d}} \int_{U} \sigma^{\epsilon} \left( a \left| D_{x} u^{\epsilon}(x,\xi,t) \right|^{2} +  \frac{1}{\tau_{\epsilon}}  \left| D_{\xi} u^{\epsilon}(x,\xi,t) \right|^{2} \right) \,dx d\xi
 \\ = \; \frac{1}{2} \int_{\mathbb{R}^{d}} \int_{U} \sigma^{\epsilon} \left( a  \left| D_{x} u_{0}^{\epsilon} \right|^{2} + \frac{1}{\tau_{\epsilon}} \left| D_{\xi} u_{0}^{\epsilon} \right|^{2} \right)\, dx d\xi\;.
\end{multline*}
Now again, by our assumption \eqref{init} on the initial condition, the right-hand side is bounded, and finally taking the supremum over $t$, we obtain
\begin{multline*}
\sup_{0 \leq t \leq T} \int_{\mathbb{R}^{d}} \int_{U} \sigma^{\epsilon} \left( a \left| D_{x} u^{\epsilon}(x,\xi,t) \right|^{2}+ \frac{1}{\tau_{\epsilon}} \left| D_{\xi} u^{\epsilon}(x,\xi,t) \right|^{2} \right)  dx d\xi \\
+ \int_{0}^{T} \int_{\mathbb{R}^{d}} \int_{U} \sigma^{\epsilon} \left| u^{\epsilon}_{t} \right|^{2} \,d\xi dx dt  \leq C\;,
\end{multline*}
which, combined with the above and the fact that $a \geq a_{0} > 0$, gives our desired estimate.
\end{proof}
\begin{proof}[Proof of Lemma \ref{lem22}]

Writing $\rho^{\epsilon} = u^{\epsilon} \sigma^{\epsilon} = \left( \left(u^{\epsilon}\right)^{2} \sigma^{\epsilon} \right)^{\frac{1}{2}} \cdot \left( \sigma^{\epsilon} \right)^{\frac{1}{2}}$, we get
\begin{equation*}
\sup_{0 \leq t \leq T} \left( \int_{\Delta} \int_{U} \left| \rho^{\epsilon} \right| \,dx d\xi \right)^{2} \;\le\; \sup_{0 \leq t \leq T} \left( \int_{\mathbb{R}^{d}} \int_{U} \left| u^{\epsilon} \right|^{2} \sigma^{\epsilon} \,dx d\xi \right) \left( \int_{\Delta} \sigma^{\epsilon}\, d\xi \right) \,\rightarrow\, 0\;.
\end{equation*}
This follows because the first term on the right-hand-side is bounded by Lemma \ref{lem21}, and because the second term on the right-hand-side goes to $0$ by \eqref{e12} and \eqref{e13}. Hence \eqref{ee2} follows.

Similarly, we deduce
\begin{equation*}
\sup_{0 \leq t \leq T} \left( \int_{\Delta} \int_{U} \left| D_{x} \rho^{\epsilon} \right| \,dx d\xi \right)^{2} \,\le\,\sup_{0 \leq t \leq T} \left( \int_{\mathbb{R}^{d}} \int_{U} \left| D_{x} u^{\epsilon} \right|^{2} \sigma^{\epsilon}\, dx d\xi \right) \left( \int_{\Delta} \sigma^{\epsilon}\, d\xi \right) \,\rightarrow\, 0
\end{equation*}
from which \eqref{ee4} follows, as well as
\begin{equation*}
\int_{0}^{T}  \left( \int_{\Delta} \int_{U} \left| \rho_{t}^{\epsilon} \right| \,dx d\xi \right)^{2}\, dt \,\le\, \int_{0}^{T} \left( \int_{\mathbb{R}^{d}} \int_{D} \left| u^{\epsilon}_{t} \right|^{2}\, \sigma^{\epsilon} \,dx d\xi \right) \left( \int_{\Delta} \sigma^{\epsilon}\, d\xi \right) dt \,\rightarrow \, 0
\end{equation*}
from which \eqref{ee32} follows.

Now define $\alpha^{\epsilon}_{i} = \alpha^{\epsilon}_{i}(x,t)$ by
\begin{equation*}
\alpha^{\epsilon}_{i}(x,t)\; :=\; \int_{\mathcal{V}_{i}} \rho^{\epsilon}(x,\xi,t) \,d\xi \;=\; \int_{\mathcal{V}_{i}} u^{\epsilon}(x,\xi,t) \, \sigma^{\epsilon}\, d\xi\;.
\end{equation*}
In the same way as above, but this time using that $\int_{\mathcal{V}_{i}} \sigma^{\epsilon} d\xi \leq \int_{\mathbb{R}^{d}} \sigma^{\epsilon} = 1$, we get
\begin{align*}
&\int_{0}^{T} \int_{U} \left| \alpha_{i}^{\epsilon} \right|\, dx dt \;\le\; \int_{0}^{T} \left( \int_{\mathcal{V}_{i}} \int_{U} \left| u^{\epsilon} \right|^{2} \sigma^{\epsilon} \,dx d\xi \right) \left( \int_{\mathcal{V}_{i}} \sigma^{\epsilon}\, d\xi \right) \,dt\; \leq \;C\;,
\\
&\int_{0}^{T} \int_{U} \left| \alpha_{i,t}^{\epsilon} \right| \,dx dt \;\le\; \int_{0}^{T} \left( \int_{\mathcal{V}_{i}} \int_{U} \left| u_{t}^{\epsilon} \right|^{2} \sigma^{\epsilon} \,dx d\xi \right) \left( \int_{\mathcal{V}_{i}} \sigma^{\epsilon}\, d\xi \right) \,dt \;\leq \;C\;,\,\mbox{and}
\\
&\int_{0}^{T} \int_{U} \left| D_{x} \alpha_{i}^{\epsilon} \right| \,dx dt \;\le\; \int_{0}^{T} \left( \int_{\mathcal{V}_{i}} \int_{U} \left| D_{x} u^{\epsilon} \right|^{2} \sigma^{\epsilon}\, dx d\xi \right) \left( \int_{\mathcal{V}_{i}} \sigma^{\epsilon}\, d\xi \right)\, dt\; \leq \;C\;.
\end{align*}
Hence for each $i$, $\left\{ \alpha^{\epsilon}_{i} \right\}$ is bounded in $H^{1}(U \times [0,T])$, a reflexive Banach space, and so by weak compactness, we can extract a subsequence $\left\{ \epsilon_{n} \right\}_{n = 1}^{\infty}$ with $\epsilon_{n} \rightarrow 0$ as $n \rightarrow \infty$, such that, for some limit functions $\alpha_{i} = \alpha_{i}(x,t)$, we have $\alpha^{\epsilon_{n}}_{i} \rightharpoonup \alpha_{i}$ weakly in $H^{1}(U \times [0,T])$ as $n \rightarrow \infty$. The results \eqref{ee1}, \eqref{ee3}, \eqref{ee31} then follow by construction.

Finally, for \eqref{ee5}, notice that
\begin{multline*}
\int_{\mathcal{V}_{i}} \int_{U} \,\left| D_{\xi} u^{\epsilon} \right|\, dx d\xi \;\le\; \left( \int_{\mathcal{V}_{i}} \int_{U} \frac{\sigma^{\epsilon}}{\tau_{\epsilon}} \left| D_{\xi} u^{\epsilon} \right|^{2}\, dx d\xi \right)^{\frac{1}{2}} \left( \int_{\mathcal{V}_{i}} \frac{\tau_{\epsilon}}{\sigma^{\epsilon}} \,d\xi \right)^{\frac{1}{2}} \\
\le \; C \left( \int_{\mathcal{V}_{i}} \frac{\tau_{\epsilon}}{\sigma^{\epsilon}} \, d\xi \right)^{\frac{1}{2}}
\end{multline*}
To show that the last integral converges to $0$ as $\epsilon\rightarrow 0$, we have
\begin{multline*}
\int_{\mathcal{V}_{i}} \frac{\tau_{\epsilon}}{\sigma^{\epsilon}} \,d\xi \;= \; \frac{1}{\epsilon}\, e^{-\frac{H-h}{\epsilon}}\, [1 + o_{\epsilon}(1)]\, e^{-\frac{h}{\epsilon}}\, (2\pi \epsilon)^{\frac{d}{2}}\, \mu \int_{\mathcal{V}_{i}} e^{\frac{\Phi}{\epsilon}}\, d\xi \\
\le  \; \frac{1}{\epsilon}\, e^{-\frac{H}{\epsilon}}\, [1 + o_{\epsilon}(1)] \, (2\pi \epsilon)^{\frac{d}{2}} \,\mu \,e^{\frac{H-\eta}{\epsilon}}\, \left| \mathcal{V}_{i} \right|
\;=\;   \frac{1}{\epsilon}\, e^{-\frac{\eta}{\epsilon}}\, [1 + o_{\epsilon}(1)] \, (2\pi \epsilon)^{\frac{d}{2}}\, \mu \left| \mathcal{V}_{i} \right|\;,
\end{multline*}
where the first identity follows from the definitions of $\tau_{\epsilon}$, $\sigma^{\epsilon}$, and $Z_\epsilon$, and \eqref{e13}. Since $\eta > 0$  the last term converges to $0$ as $\epsilon\rightarrow 0$.
Therefore, it follows that, on $\mathcal{V}_{i} \times U$, $u^{\epsilon} \rightarrow u_{i}$ a.e. for some function $u_{i} = u_{i}(x,t)$. But using $\rho^{\epsilon} = \sigma^{\epsilon} u^{\epsilon},$ integrating with respect to $\xi$ on $\mathcal{V}_{i}$ and using $\int_{\mathcal{V}_{i}} \sigma^{\epsilon} \,d\xi = \hat{\mu_{i}}$, we finally obtain $u_{i} = \frac{\alpha_{i}}{\hat{\mu_{i}}}$.
\end{proof}

\smallskip\noindent{\bf Acknowledgments.} The authors wish to thank Professor Lawrence C. Evans for his fruitful discussions. The research of I. Seo is supported by the National Research Foundation of Korea NRF grant funded by the Korean government MSIT (Project 2018R1C1B6006896).

\end{document}